\documentclass[oneside,final,onefignum,onetabnum]{siamart220329}
  
\usepackage{amsfonts, amssymb, bm, amsmath, cases}
\usepackage{amsmath}
\usepackage{amssymb}
\usepackage{amsfonts}
\usepackage{mathrsfs} 
\usepackage{color}

\usepackage{graphicx}
\usepackage{color}
\usepackage{bm}

\usepackage[short,nocomma]{optidef}
\usepackage{booktabs}
\usepackage{cancel}

\usepackage{amsfonts}
\usepackage{booktabs}
\usepackage{siunitx}
\usepackage{icomma}

\usepackage{lscape}
\usepackage[paper=portrait,pagesize]{typearea}
\usepackage{multirow}
\usepackage{cancel}

\usepackage{cleveref}

\def\Z{\mathbb{Z}}\def\Q{\mathbb{Q}}
\def\R{\mathbb{R}}
\def\C{\mathbb{C}}

\def\pset{{{P}}}

\def\detv{{\mathscr{D}}}
\def\inc{{\mathscr{V}}}
\def\incproj{{\mathscr{W}}}
\def\calF{{\mathcal{F}}}

\def\calA{{\mathcal{A}}}

\def\size{n}
\def\cons{m}
\def\rnk{r}
\def\affmap{\mathscr{A}}
\def\face{\mathcal{F}}
\def\sym{\mathbb{S}}

\newtheorem{assumption}{Assumption}

\newcommand\hybrid{{\tt{Hybrid}}}
\newcommand\msolve{{\tt{msolve}}}

\newcommand\HNS{{\tt{HNS}}} 

\newcommand\Macaulay{{\tt{Macaulay2}}}

\newcommand\la{\left\langle}
\newcommand\ra{\right\rangle}
\newcommand\rank{\text{rank}}

\newcommand\pmtx[1]{\begin{pmatrix}#1\end{pmatrix}} 
 
\newcommand\smtx[1]{\left(\begin{smallmatrix}#1\end{smallmatrix}\right)}

\newcommand\mymid{\,\, : \,\,}

\newcommand{\rev}[1]{{#1}}

\newcommand\subsetsim{\mathrel{%
  \ooalign{\raise0.2ex\hbox{$\subset$}\cr\hidewidth\raise-0.8ex\hbox{\scalebox{0.9}{$\sim$}}\hidewidth\cr}}}

\def\myparagraph#1{\vspace{2pt}{\bf #1.}}

\usepackage[a4paper,
            bindingoffset=0in,
            left=1in,
            right=1.55in,
            top=1in,
            bottom=1in,
            footskip=.25in]{geometry}

\usepackage{braket,amsfonts}

\usepackage{array}

\usepackage[caption=false]{subfig}

\usepackage{pgfplots}

\newsiamthm{claim}{Claim}
\newsiamremark{remark}{Remark}
\newsiamremark{hypothesis}{Hypothesis}
\crefname{hypothesis}{Hypothesis}{Hypotheses}

\usepackage{algorithmic}

\usepackage{graphicx,epstopdf}

\Crefname{ALC@unique}{Line}{Lines}

\usepackage{amsopn}

\usepackage{xspace}
\usepackage{bold-extra}
\usepackage[most]{tcolorbox}

\colorlet{texcscolor}{blue!50!black}
\colorlet{texemcolor}{red!70!black}
\colorlet{texpreamble}{red!70!black}
\colorlet{codebackground}{black!25!white!25}

\lstdefinestyle{siamlatex}{%
  style=tcblatex,
  texcsstyle=*\color{texcscolor},
  texcsstyle=[2]\color{texemcolor},
  keywordstyle=[2]\color{texemcolor},
  moretexcs={cref,Cref,maketitle,mathcal,text,headers,email,url},
}

\tcbset{%
  colframe=black!75!white!75,
  coltitle=white,
  colback=codebackground, 
  colbacklower=white, 
  fonttitle=\bfseries,
  arc=0pt,outer arc=0pt,
  top=1pt,bottom=1pt,left=1mm,right=1mm,middle=1mm,boxsep=1mm,
  leftrule=0.3mm,rightrule=0.3mm,toprule=0.3mm,bottomrule=0.3mm,
  listing options={style=siamlatex}
}

\newtcblisting[use counter=example]{example}[2][]{%
  title={Example~\thetcbcounter: #2},#1}

\newtcbinputlisting[use counter=example]{\examplefile}[3][]{%
  title={Example~\thetcbcounter: #2},listing file={#3},#1}

\DeclareTotalTCBox{\code}{ v O{} }
{ 
  fontupper=\ttfamily\color{black},
  nobeforeafter,
  tcbox raise base,
  colback=codebackground,colframe=white,
  top=0pt,bottom=0pt,left=0mm,right=0mm,
  leftrule=0pt,rightrule=0pt,toprule=0mm,bottomrule=0mm,
  boxsep=0.5mm,
  #2}{#1}

\patchcmd\newpage{\vfil}{}{}{}
\usepackage{hyperref}
\hypersetup{
    colorlinks=true,
    citecolor=blue,
    linkcolor=blue,
    filecolor=magenta,      
    urlcolor=blue,
    pdftitle={papersioptknz}
    }

\begin{tcbverbatimwrite}{tmp_\jobname_header.tex}
  
  \title{Certifying solutions of degenerate semidefinite programs\thanks{Accepted to SIAM J. Optimization (April 2025)\funding{Research was supported by the Doctoral Programme Vienna Graduate School on Computational Optimization (VGSCO) which was funded by FWF (Austrian Science Fund), project W 1260-N35. This work has been partially supported by the ANR JCJC project ANR-21-CE48-0006-01 “HYPERSPACE”.}}}
  
  \author{Vladimir Kolmogorov\thanks{Institute of Science and Technology Austria (ISTA) (\email{vnk@ist.ac.at}).}
    \and Simone Naldi\thanks{Université de Limoges, CNRS, XLIM, UMR 7252, F-87000 Limoges, France\\ \hbox{}\hspace{0.45cm} Sorbonne Université, CNRS, LIP6, F-75005 Paris, France. (\email{simone.naldi@unilim.fr}).}
    \and Jeferson Zapata\thanks{Institute of Science and Technology Austria (ISTA) (\email{jeferson.zapata@ist.ac.at}).}}
  
  \headers{Verifying feasibility of degenerate semidefinite programs}{V. Kolmogorov, S. Naldi, and J. Zapata}
\end{tcbverbatimwrite}

\input{tmp_\jobname_header.tex}

\pgfplotsset{compat=1.18}
\begin{document}

\maketitle

\begin{tcbverbatimwrite}{tmp_\jobname_abstract.tex}
  \begin{abstract}
   This paper deals with the algorithmic aspects of solving feasibility problems of semidefinite programming (SDP), aka linear matrix  inequalities (LMI). Since in some SDP instances all feasible solutions have irrational entries, numerical solvers that work with rational numbers can only find an approximate solution. We study the following question: is it possible to certify feasibility of a given SDP using an approximate solution that is sufficiently close to some exact solution? Existing approaches make the assumption that there exist rational feasible solutions (and use techniques such as rounding and lattice reduction algorithms). We propose an alternative approach that does not need this assumption. More specifically, we show how to construct a system of polynomial equations whose set of real solutions is guaranteed to have an isolated correct solution (assuming that the target exact solution is maximum-rank). This allows, in particular, to use algorithms from real algebraic geometry for solving systems of polynomial equations, yielding a hybrid (or symbolic-numerical) method for SDPs. We experimentally compare it with a pure symbolic method in~\cite{HNS2015c}; the hybrid method was able to certify feasibility of many SDP instances on which~\cite{HNS2015c} failed. \rev{Our approach may have further applications,} such as refining an approximate solution using methods of numerical algebraic geometry for systems of polynomial equations.
  \end{abstract}

  
\end{tcbverbatimwrite}
\input{tmp_\jobname_abstract.tex}


\section{Introduction}
Semidefinite programming (SDP) is a nontrivial generalization of linear programming (LP) asking to minimize linear functions defined over the space of real symmetric matrices and restricted to the PSD cone (the convex cone of positive semidefinite matrices). Despite its simple definition the complexity analysis of SDP and of its feasibility problem (LMI, linear matrix inequalities) is still a fairly open question, and likewise the development of efficient but reliable algorithms, which is at the heart of research in convex optimization.

Beyond its interest as conic optimization problem, the importance of SDP is enhanced by its numerous applications, the most classical of which is perhaps the semidefinite relaxation of MAX-CUT by Goemans and Williamson \cite{goemans}. More generally,  SDP is used as numerical tool for solving hard non-convex optimization problems through the so-called moment-SOS hierarchy \cite{henrion2020moment}, consisting in relaxing polynomial optimization to a sequence of SDP problems of increasing size and with good convergence properties \cite{lasserre01globaloptimization}. Specific applications include checking stability of systems in control theory~\cite{Papachristodoulou:05}
and analyzing the convergence rate of numerical algorithms for convex optimization~\cite{Drori:14}. In both cases one needs to find a certain Lyapunov function,
and the search for such function is cast as a small-scale SDP.

Unfortunately, naively relying on numerical SDP solvers may lead to incorrect conclusions; examples for small dynamical systems can be found in~\cite{roux2018validating}.
This motivates the need for algorithms that can certify feasibility of a given SDP. Below we discuss two classes of such algorithms:
 {\em symbolic} methods that compute an exact solution in a {\em Rational Univariate Representation} using techniques from real algebraic geometry,
and {\em symbolic-numerical} (or {\em hybrid}) methods that certify feasibility using a numerical SDP solver as a subroutine.

\myparagraph{Symbolic methods}
Because the constraints in SDP are nonlinear in the entries of the matrix, but still algebraic (defined by polynomial inequalities), the solution of a SDP program is in general not rational but defined over some algebraic extension of the base field, which we usually assume to be $\mathbb{Q}$. The degree of this extension, known as the algebraic degree of SDP, represents a measure of the inner complexity of the program, for which exact formulas are known \cite{stu}.

The underlying algebraic structure of SDP has motivated in the last years the development of computer algebra algorithms whose arithmetic complexity is essentially quadratic in some multilinear bound on the mentioned algebraic degree \cite{HNS2015c,naldi:tel-01212502}. The strategy of the series of work \cite{HNS2014,HNS2015c,HNS2015b,naldiIssac2016integral} relies on reducing the problem of finding a matrix in a section of the PSD cone to the one of finding low rank elements in some affine space of matrices. This can be cast as a real-root finding problem, asking to compute at least one point in every real connected component of an algebraic set, modeled using systems of polynomial equations with finitely-many solutions and hence solved using Gr\"obner-bases-based algorithms \cite{berthomieu2021msolve}.  By its purely symbolic nature, these algorithms are not comparable, in terms of target problem size, with numerical methods for SDP that are mainly based on variants of interior-point method (IPM), and their correctness depend on genericity assumptions on input data. 

Concerning algebraic approaches to SDP, let us mention the method in \cite{naldi2020conic} based on homogenization and projective geometry for identifying the feasibility type of a LMI and the certificate in \cite{klep2013exact} based on the theory of sums of squares polynomial for checking SDP infeasibility.

\myparagraph{Hybrid methods} These methods can potentially handle much larger SDP instances. Most of the existing work has focused on the case when the set of feasible solutions
contains rational points~\cite{harrison2007verifying,KLYZ12,monniaux2011generation,PaPe,Platzer:09,roux2018validating,dostert2020exact}. Some of these papers~\cite{monniaux2011generation,dostert2020exact} employ a lattice reduction technique, namely the Lenstra-Lenstra-Lov\'asz (LLL) algorithm. The paper~\cite{dostert2020exact}
discusses an extension to quadratic fields $\mathbb Q[\sqrt\ell]$, but the method seems to assume that $\ell$ is known.

Let us finally mention the method by numerical algebraic geometry in \cite{BertiniSDP} which does not strictly fall into the two classes above. It is based on SDP duality and applies homotopy continuation for numerically tracking the central path.

\myparagraph{Our contributions} In this paper we develop an alternative hybrid method which is not restricted to the case when the feasible region contains rational points.
The method essentially reduces the problem of certifying feasibility of an SDP to that of solving a system of polynomial equations 
whose set of real solutions has a zero-dimensional component with a desired solution.
Our main result can be stated as follows (a more refined formulation is given later in Theorems~\ref{th:method} and~\ref{th:numerical}).

\begin{theorem} There exists an algorithm that, given a feasible system 
$\affmap(X) = b, \, X \succeq 0$,  approximate solution $\tilde X$ \rev{and parameter $\epsilon>0$}, constructs a system of polynomial equations over symmetric matrices $X\in\mathbb R^{n\times n}$ 
and auxiliary variables $Y\in\mathbb R^{\size \times (\size-\rnk)}$
of the following form:
\begin{subequations}\label{eq:GALKDJHGALKH}
\begin{align}
\affmap(X) &\;\;=\;\; b \;,\;\;\; X^\top \;\;=\;\; X \label{eq:GALKDJHGALKH:a}
\\
\Pi^\top X \Pi \left(\begin{matrix} Y \\ I_{\size-\rnk}\end{matrix}\right)  &\;\;=\;\; 0^{\size \times (\size-\rnk)}  \label{eq:GALKDJHGALKH:b}
\\
X_{ij}&\;\;=\;\; \tilde X_{ij} \qquad \forall (i,j)\in K \label{eq:GALKDJHGALKH:c}
\end{align}
\end{subequations}
where $\Pi$ is a permutation matrix, $r\in[n]$, and $K$ is a subset of $\{(i,j)\in [n]\times[n]\::\:i\le j\}$.
(Quantities $\Pi$, $r$ and $K$ are computed from $\affmap$, $b$, $\tilde X$ \rev{and $\epsilon$}).
Furthermore,  the set of real solutions of system~\eqref{eq:GALKDJHGALKH} has a zero-dimensional component $\{(X,Y)\}$ with $X\in P := \{X \succeq 0\::\:\affmap(X) = b\}$
\rev{assuming that $P$ contains a maximum-rank solution $X^\ast$ such that \\
{\rm (i)}  parameter $\epsilon$ is sufficiently small:  $\epsilon\le\epsilon_{\max}$
where $\epsilon_{\max}>0$ depends on $\calA$, $b$, $X^\ast$; \\
{\rm (ii)} $\tilde X$ is sufficiently close to $X^\ast$:
 $||\tilde X-X^\ast||\le \delta$ where $\delta>0$ depends on $\calA$, $b$, $X^\ast$, $\epsilon$.}
\end{theorem}

\rev{A precise definition of $\delta$ is given in the proof of \Cref{th:method}. 
Concerning $\epsilon_{max}$, this is a tolerance value that allows to numerically compute
maximal sets of linearly independent columns in Steps 2 and 4 of our main \Cref{alg:method}.}

There may be several ways to exploit this result. One possibility could be to refine the approximate solution $\tilde X$ and the associated matrix $\tilde Y$
(which is easily computable from $\tilde X$) by applying the Newton method for system~\eqref{eq:GALKDJHGALKH}.\footnote{\rev{The formula
for computing $\tilde Y$ from $\tilde X$ will be specified later, in Algorithm~\ref{alg:method}. The mapping $\tilde Y(\tilde X)$
will satisfy the following properties: (i) it is continuous in some open
neighborhood of $\tilde X$; (ii) 
matrix $Y=\tilde Y(X^\ast)$ is the unique solution of~\eqref{eq:GALKDJHGALKH:b}
for fixed $X=X^\ast$.}}
For some polynomial systems this method converges very slowly near the optimum or even diverges; in that case one could try to solve~\eqref{eq:GALKDJHGALKH} using techniques from numerical algebraic geometry such as those implemented in the software Bertini~\cite{BertiniBook}. The approximate solution $(\tilde X,\tilde Y)$ could allow the method to focus on the desired zero-dimensional component, and use the regime known as the {\em endgame}.

In this paper we explore a different direction: we solve system~\eqref{eq:GALKDJHGALKH} using exact methods from real algebraic geometry, similar to~\cite{HNS2015c}.
Accordingly, our experimental results focus on a comparison with~\cite{HNS2015c}. Note that the latter paper
also reduces the problem to a system of polynomial equations of the form~\eqref{eq:GALKDJHGALKH:a}-\eqref{eq:GALKDJHGALKH:b}, but with the following differences:
(1) the algorithm in \cite{HNS2015c} has no a priori information about a feasible rank and thus needs to exhaustively search over all 
(exponentially many) kernel profiles until a solution is found;
(2) it provides a guarantee that for one of these kernel profiles the set of real solutions contains a (possibly positive-dimensional) connected component whose points $(X,Y)$ satisfy $X\in P$;
(3) it searches for a minimum-rank solution, whereas we search for a maximum-rank solution.

As in~\cite{HNS2015c}, we use a critical point method to find one real solution per component. This method works under certain assumptions
(such as complete intersection of the corresponding variety or finiteness of the number of solutions of some critical equations);
on the instances that we tested 
these assumptions were often violated for the method in~\cite{HNS2015c} but not for our hybrid method,
and accordingly the hybrid method was able to certify feasibility of many instances on which~\cite{HNS2015c} failed.

\myparagraph{Outline} The paper is organized as follows. \Cref{sec:prelim} contains the main theoretical ingredients and notations. In \Cref{sec:incidence} we make a link between determinantal varieties and facial reduction for SDP. \Cref{sec:descMethod} describes a hybrid algorithm, which is applied to a collection of instances (listed in \Cref{appA}) of degenerate SDPs of the literature in \Cref{sec:numerical}. An example is developed in full details in \Cref{sec:DruWoExample}.


\section{Preliminaries}
\label{sec:prelim}
\phantom{}

\myparagraph{Notation for matrices} 
All norms $||\cdot||$ for vectors and matrices used in this paper are 2-norms.
The Frobenius norm for matrices is denoted by $||\cdot||_F$.

For a matrix $A\in\mathbb R^{p\times q}$ we denote by
$\sigma_i(A)\ge 0$ the $i$-th singular value of $A$; if $i>\min\{p,q\}$ then $\sigma_i(A)=0$ by definition.
We denote by $\sym^p_{\mathbb F}$ the set of $p\times p$ symmetric matrices with coefficients in a field ${\mathbb F}$ and for ${\mathbb F}=\R$ we simply write $\sym^p$.
If $A$ is real symmetric of size $p$ ($A \in \sym^p$) then $\lambda_i(A)$ denotes the $i$-th largest eigenvalue of $A$.
A matrix $A \in \sym^p$ is positive semidefinite ($A \succeq 0$) if $\lambda_p(A) \geq 0$ and positive
definite ($A \succ 0$) if $\lambda_p(A)>0$. Recall that if $A \succeq 0$ then $\lambda_i(A)=\sigma_i(A)$ for all $i$.

The following inequalities are well known (see e.g.~\cite[Section 8]{golub2013matrix}):
\begin{subequations}
\begin{eqnarray}
|\sigma_i(A)-\sigma_i(B)| & \le & ||A-B||\qquad \forall A,B\in \mathbb R^{p\times q},\; \forall i \label{eq:sigmaInequality} \\
|\lambda_i(A)-\lambda_i(B)| & \le & ||A-B||\qquad \forall A,B\in \sym^p,\; \forall i \label{eq:lambdaInequality}
\end{eqnarray}
\end{subequations}
We will also denote $\rho(A)=\sigma_r(A)$ where $r=\text{rank}(A)$.
Clearly, for every non-zero matrix $A$ we have $\rho(A)>0$.

For a set of rows $I\subseteq[p]$ let $A_I\in\mathbb R^{|I|\times q}$ be the submatrix of $A$ formed by the rows in $I$. 
Similarly, for a set of columns $J\subseteq[q]$ let $A^J\in\mathbb R^{p\times|J|}$ be the submatrix of $A$ formed by the columns in $J$.

The transpose of matrix $A$ is denoted by $A^\top$.
In this paper this operation is applied only to matrices with real entries. We never use conjugate transpose for complex matrices, and accordingly we reserve the symbol ``$^\ast$'' for a different purpose ($X^\ast$ will denote an exact solution of the SDP system).

\myparagraph{Semidefinite programming}
Let us equip the vector space $\sym^\size$ of real symmetric matrices of size $\size$ with the Frobenius inner product:
\[
\la M, N \ra_{F} := \text{Trace}(MN)=\sum_{i,j} M_{ij}N_{ij}.
\]
for $M=(M_{ij}), \, N=(N_{ij}) \in \sym^\size$. The cone of positive semidefinite matrices (PSD cone) is $\sym^\size_+ = \{M \in \sym^\size : M \succeq 0\}$ is a convex closed basic semialgebraic set in $\sym^\size$, and its interior is the open convex cone of positive definite matrices: $\sym^\size_{++} := \text{int}(\sym^\size_+)\subseteq \sym^\size$. We denote by $\la u, v \ra_{\R} = \sum u_iv_i$ the Euclidean inner product on $\R^\cons$.

A semidefinite program in standard primal form is given by
\begin{equation}
  \label{SDPp}
\begin{array}{rcll}
  p^* & := & \inf_{X \in \sym^\size} & \la C, X \ra_{F} \\
  &    & \text{s.t.}         & \affmap(X) = b \\
  &    &                     & X \in \sym^\size_+
\end{array}
\end{equation}
where $\affmap \mymid \sym^\size \to \R^\cons$ is a linear map defined by $\affmap(X) = (\la A_{1},
X\ra_F, \ldots, \la A_{m}, X\ra_F)$ for some real symmetric matrices
$A_{1}, \ldots, A_{m} \in \sym^\size,$ $b \in \R^\cons$, and $C \in \sym^\size$.
The feasible sets of the program \eqref{SDPp} is denoted by
\begin{equation}\label{feassetp}
\pset := \{X \in \sym^\size \mymid \affmap(X)=b, X \succeq 0\}.
\end{equation}

Problem \eqref{SDPp} with $P\ne\varnothing$ is called {\em strongly feasible} if $P$ contains a matrix $X$ with $X\succ 0$, and {\em weakly feasibly}
otherwise. We will be mainly interested in weakly feasible SDPs. (Note that if~\eqref{SDPp} is strongly feasible then $P$ contains rational-valued matrices.) 

\myparagraph{Algebraic and semialgebraic geometry}
We refer to \cite{CLO} for basics of (computational) algebraic geometry, and we recall here the main definitions
needed in this work. A (complex) algebraic variety is a set of the form 
\begin{equation}\label{variety}
V = \{x \in \C^n : f_1(x) = 0, \ldots, f_s(x)=0\}
\end{equation}
for polynomials $f_1,\ldots,f_s \in {\mathbb F}[x]$ with coefficients in some subfield ${\mathbb F} \subseteq \C$ 
(we usually fix ${\mathbb F}=\Q$ or ${\mathbb F}=\R$). Of course defining $I = \langle f_1,\ldots,f_s \rangle \subseteq {\mathbb F}[x]$ the ideal generated by $f_1,\ldots,f_s$, 
remark that every $f \in I$ vanishes over $V$, in other words $V = V(I)$ with 
$$V(I) := \{x \in \C^n : f(x)=0, \forall f \in I\}.$$
The vanishing ideal of a set $W \subseteq \C^n$ is defined as
$$
I(W) := \{f \in \C[x] : f(x) = 0, \, \forall x \in W\}
$$
and by Hilbert's Nullstellensatz \cite[\S 4.1]{CLO}, one has the correspondence $I(V(J)) = \sqrt{J}$, where $\sqrt{J} := \{f \in \C[x] : 
f^m \in J, \exists m \in \mathbb{N}\}$ is the radical ideal of $J$. An ideal $I$ is called radical if $I=\sqrt{I}$.
The smallest algebraic variety containing a set $W$ is denoted by $\overline{W}$ and called its Zariski-closure. 

A variety $V \subseteq \C^n$ is called irreducible if it is not the union of two proper subvarieties, and every variety
is finite union of irreducible varieties called its irreducible components.
The dimension of an algebraic variety $V \subseteq \C^n$ is the minimum integer $d = \dim(V)$ such that the intersection 
$V \cap H_1 \cap \cdots \cap H_d$ of $V$ with $d$ generic hyperplanes $H_i$ is finite.
The degree of an irreducible variety $V \subseteq \C^n$ is the cardinality of the intersection of $V$ with $\dim(V)$
generic hyperplanes, and the degree of any variety is the sum of the degrees of its irreducible components. 
A variety whose irreducible components have the same dimension is called equidimensional.
An ideal $I$ with $V(I)\subseteq \C^n$ is in complete intersection if it can be generated by $\text{codim}(V(I))
= n - \dim(V(I))$ many polynomials, and with abuse of notation $V$ is said in complete intersection if this is the case
for $I(V)$.

Let $V \subseteq \C^n$ be a variety with $I(V)=\langle f_1, \ldots, f_s \rangle$.
For a point $p\in V$ the tangent space $T_pV$ at $p$ to $V$ is the kernel of the Jacobian matrix 
$$
J(f_1,...,f_s) := \left(\frac{\partial f_i}{\partial x_j}\right)_{\begin{smallmatrix} 1 \leq i \leq s \\ 1 \leq j \leq n \end{smallmatrix}}
$$ 
evaluated at $p$. A point \( p \in V \) is said to be a \textit{regular point} if the dimension of the tangent space \( T_p V \) at \( p \) coincides with the local dimension of \( V \) at \( p \) (the maximum of the dimensions of the irreducible components containing $p$). 
When $V$ is equidimensional of dimension $d$, the Jacobian 
$J(f_1,...,f_s)$ has rank $\leq n-d$ and equality holds exactly at regular points.
We denote the set of regular points as \( \text{Reg}(V) \), and its complement, \( \text{Sing}(V) := V \setminus \text{Reg}(V) \), is referred to as the set of singular points of \( V \). 

For ${\mathbb F}=\R$ and $V$ as in \eqref{variety}, the set $V(\R) := V \cap \R^n$ is called a real algebraic variety, and it
is defined by the vanishing of finitely many real polynomials. Given $g_1,\ldots,g_t \in \R[x]$, the set
$$
S = \{x \in \R^n : g_1(x) \geq 0, \ldots, g_t(x) \geq 0\}
$$
is called a basic closed semialgebraic set.

\myparagraph{Determinantal and incidence varieties}
We refer to \cite{eisenbud1988linear,harris1984symmetric} for the general theory of determinantal
varieties and we recall below the main definitions in the context of LMI.

Let $A_1,\ldots,A_\cons \in \sym^\size$, be the matrices defining the map $\mathscr{A}$
and $b \in \R^\cons$ the vector in \Cref{SDPp}. The set
$$
\detv_\rnk := \{X \in \sym^\size_\C \mymid \rank(X) \leq r, \affmap(X) = b\}
$$
is called the determinantal variety associated with $\mathscr{A},b$ and $r$.
It is the algebraic variety defined by the vanishing of the $(r+1)\times(r+1)$ minors of
the matrix $X$ (plus linear constraints $\mathscr{A}(X)=b$). 
For some $\rnk \in \{0,1,\ldots,m\}$, the real variety $\detv_\rnk(\R) := \detv_\rnk \cap \sym^n$ intersects the 
feasible set $P$.

From a computational point of view, minors are not easy to handle. Instead, the rank constraint is equivalent to the existence 
of a matrix of high rank whose columns generate the kernel of the original matrix. The following set
$$
\inc_r := \{(X,Y) \in \sym^\size_\C \times \C^{\size \times (\size-\rnk)} \mymid \affmap(X)=b, X Y = 0, \rank(Y) = \size-\rnk\}
$$
is associated with the variety $\detv_r$; indeed, $\detv_r$ is the projection of $\inc_r$ to $\sym^\size_\C$. 
Remark that $\inc_r$ is not an algebraic variety, since the constraint on
the rank of $Y$ is not closed in the Zariski topology, but instead a constructible set (an open subset of a Zariski 
closed set). We use the following view of $\inc_r$ from the coordinate charts (see \cite{HNS2015c}):
\begin{equation}\label{sssec:det_inc}
\inc_{r,\iota} := \{(X,Y) \in \sym^\size_\C \times \C^{\size \times (\size-\rnk)} \mymid \affmap(X)=b, X Y = 0, Y_{[n]-\iota} = I_{\size-\rnk}\}
\end{equation}
for $\iota \subseteq [\size]=\{1, \ldots, \size\}$, of cardinality $|\iota| = \rnk$, and $Y_{\iota'}$ is the
matrix obtained by selecting the rows of $Y$ in $\iota'=[n]-\iota$. Note that $\inc_{\rnk,\iota}$ is an algebraic set,
and  $\inc_{r} = \cup_{|\iota|=r} \inc_{r,\iota}$.
For a field ${\mathbb F}\in\{\R,\C\}$ and subset $C\subseteq \sym^\size_{\mathbb F} \times {\mathbb F}^{\size \times (\size-\rnk)}$ let us denote $\pi_{\sym^\size_{\mathbb F}}(C)=\{X\in\sym^\size_{\mathbb F}\::\:\exists Y\in {\mathbb F}^{n\times(n-r)}\mbox{ s.t.\ }(X,Y)\in C\}$ to be the projection of $C$ to $\sym^\size_{\mathbb F}$.
Using this notation, we define set
$$
\incproj_{\rnk,\iota}:=\pi_{\sym^\size_\C}(\inc_{\rnk,\iota})
$$
We also define the real traces of 
 $\inc_{\rnk,\iota}$ and $\incproj_{\rnk,\iota}$ as follows:
\begin{align*}
\inc_{\rnk,\iota}(\R) &= \inc_{\rnk,\iota} \cap (\sym^\size \times \R^{\size \times (\size-\rnk)})  \\
\incproj_{\rnk,\iota}(\R) &= \incproj_{\rnk,\iota} \cap \sym^\size 
=\pi_{\sym^\size}(\inc_{\rnk,\iota}(\R))
\end{align*}
where the last equality holds due to the following observation:
if $(X,Y)\in \inc_{r,\iota}$ with $X\in\sym^\size$ then $(X,{\mathscr Re}\, Y)\in \inc_{r,\iota}(\R)$.
Note that $\incproj_{\rnk,\iota}(\R)$ may not be algebraic; it is a constructible set. 

Incidence varieties are used in \cite{HNS2015c} to solve feasibility problems of generic instances of LMI, and the following Theorem is a refinement of results in \cite{HNS2015c}.

\begin{theorem}\label{th:HNS2}
Assume $P \neq \varnothing$. Let $X^\ast\in \sym^n_+$ be a minimum-rank solution of~\eqref{feassetp}, and let $\iota\subseteq[n]$
be a maximal set of linearly independent columns of $X^\ast$, with $r=|\iota|$. 
\begin{itemize}
    \item[(a)] The constructible set $\incproj_{\rnk,\iota}(\R)$ has a connected component $C$ with $X^\ast\in C\subseteq \pset$.
    \item[(b)] The real variety $\inc_{r,\iota}(\R)$ has a connected component $C'$ with $X^\ast\in\pi_{\sym^n}(C')\subseteq \pset$
\end{itemize}
\end{theorem}
\begin{proof}
The definition of $\iota$ implies that there exists $Y^\ast\in\mathbb R^{n\times(n-r)}$ with $X^\ast Y^\ast=0$ and $Y^\ast_{[n]-\iota}=I_{n-r}$.
Clearly, we have $(X^\ast,Y^\ast)\in \inc_{r,\iota}(\R)$ and hence $X^\ast \in \incproj_{\rnk,\iota}(\mathbb{R})$.
Let $C$ be the component of $\incproj_{\rnk,\iota}(\mathbb{R})$ containing $X^\ast$.

\myparagraph{(a)} We need to show that $C\subseteq P$.
If $r=0$ then all points $(X,Y) \in \inc_{r,\varnothing}(\R)$ satisfy $X=0$, hence $C=\{X^\ast\}=\{0\}$ and the claim is straightforward. Suppose that $r>0$
and the claim is false, then there exists continuous curve $\{X^{(t)}\}_{t\in[0,1]}\subseteq C\subseteq \incproj_{\rnk,\iota}(\mathbb{R})$
with $X^{(0)}=X^\ast$ and $X^{(1)}=X\notin P$. For each $t\in[0,1]$ there must exist matrix $Y^{(t)}$ with
$(X^{(t)},Y^{(t)})\in \inc_{r,\iota}(\R)$.
Denote $\lambda^{(t)}_i=\lambda_i(X^{(t)})$ for $t\in[0,1]$, so that $\lambda^{(t)}_1\ge\ldots\ge\lambda^{(t)}_n$. These values satisfy the following properties for each $t\in[0,1]$:
\begin{itemize}
\item[(i)] At least $n-r$ values in $\lambda^{(t)}_1,\ldots,\lambda^{(t)}_n$ are zeros,
or equivalently $\rank(X^{(t)})\le r$.
This follows from conditions $X^{(t)}Y^{(t)}=0$ and $\rank(Y^{(t)})=n-r$, which imply that $\rank(X^{(t)})\le n-\rank(Y^{(t)})\le n-(n-r)=r$.
\item[(ii)] If $\lambda^{(t)}_n\ge 0$ then $\lambda^{(t)}_{r+1}=\ldots=\lambda^{(t)}_{n}=0$ and $\lambda^{(t)}_r>0$.
The first claim follows from~(i); if $\lambda^{(t)}_r=0$ then $\rank(X^{(t)})\le r-1$ contradicting the minimality of $r$
(note that $X^{(t)}\succeq 0$ and hence $X^{(t)}\in P$).
\item[(iii)] For each $i\in[n]$, $\lambda^{(t)}_i$ is a continuous function of $t$. This holds by continuity of the curve and by eq.~\eqref{eq:lambdaInequality}.
\end{itemize}
Conditions $X^{(0)}\in P$ and $X^{(1)}\notin P$ give $\lambda_n^{(0)}\ge 0$ and $\lambda_n^{(1)}< 0$.
Denote $s=\sup \{t\in[0,1]\::\:\lambda_n^{(t)}\ge 0\}$,
then $\lambda_n^{(s)}\ge 0$ by continuity and hence $s<1$. 
By (ii) we have $\mu := \lambda_r^{(s)}>0$.
By continuity, there exists $u\in(s,1]$ such that $\lambda_r^{(t)}>\mu/2$ for all $t\in[s,u]$.
Condition (i) thus gives  $\lambda_{r+1}^{(t)}=\ldots=\lambda_{n}^{(t)}=0$ for all $t\in[s,u]$,
and thus $\sup \{t\in[0,1]\::\:\lambda_n^{(t)}\ge 0\}\ge u>s$, which is a contradiction.

\myparagraph{(b)} Let $C'$ be the connected component of $\inc_{r,\iota}(\mathbb{R})$ containing $(X^\ast,Y^\ast)$.
Clearly, we have $\pi_{\sym^n}(C')\subseteq C$, and so the claim holds by part (a).
\end{proof}


\section{Incidence varieties and facial reduction}
\label{sec:incidence}

We consider the feasibility problem of the primal SDP in \eqref{SDPp}, that is the system:
\begin{equation}
\label{FEASp}
\affmap(X) = b, \,\,\,\,\,\,\,\, X \succeq 0.
\end{equation}
Let $P = \{X \in \sym^\size \mymid \affmap(X)=b, X \succeq 0\}$ be the set of feasible solutions as in~\eqref{feassetp},
and let $\rnk$ be the maximum rank of a matrix in $P$.
From now on, we assume that $P\ne\varnothing$ and $r>0$.
A matrix $X\in P$ of rank $r$ will be called a {\em maximum-rank solution}.

Below we recall some basic facts about facial reduction for SDPs, see e.g.~\cite{drusvyatskiy2017many}.
Let~$\face$ be the minimal face of the SDP cone $\sym^\size_+$ containing $P$. 
It can be represented by a full-rank
matrix $U \in \R^{\size \times \rnk}$ as follows: 
\begin{equation}\label{eq:face-range}
\face = \{X \in \sym^\size_+\::\: \text{range}(X)\subseteq \text{range}(U)\}=\{U Z U^\top\::\: Z \in \sym^r_+\}.
\end{equation}
In other words, $\face$ is linearly isomorphic (through $UZU^\top \to Z$) to a copy of the cone
$\sym^\rnk_+$, and in particular it has dimension $\binom{\rnk+1}{2}$. Any matrix in the relative
interior of $P$ is in the relative interior of $\face$ and hence has rank $\rnk$; every other
matrix in $P$ has rank $< \rnk$.

Let $\iota\subseteq[n]$ be a maximal set of linearly independent rows of $U$, with $|\iota|=r$.
We assume for notational convenience that $\iota=[r]$
(this can be achieved by renaming rows and columns of $X$).
Since $U$ is full-rank, we can apply Gaussian elimination to compute the column reduced
echelon form of $U$, without changing its range. 
Accordingly, we can assume w.l.o.g.\ that
\begin{equation}
U=\pmtx{I_r \\ U_0},\qquad \text{with } U_0\in \R^{(n-r) \times r}. \label{eq:Uform}
\end{equation}

Then for the matrix $Y_U := -{U_0}^\top$ we have
\begin{subequations}
  \label{eq:face}
  \begin{align}
    \face =\{X \in \sym^\size_+ \mymid \text{range}(X) \subseteq \text{range}(U)\}
    &=\{X\in \sym^\size_+ \mymid \smtx{Y_U^\top & I_{\size-\rnk}}
    X = 0^{(\size-\rnk)\times \size}\} \label{eq:face:a} \\
    &=\{X\in \sym^\size_+ \mymid X \smtx{Y_U \\ I_{\size-\rnk}}=0^{\size
      \times (\size-\rnk)}\} \label{eq:face:b}
  \end{align}
\end{subequations}

Now  consider the following system of equations (in variables $X \in \sym^\size_\C$ and $Y \in \C^{\size \times (\size-\rnk)}$):
\begin{subequations}\label{eq:GHOAISFH}
\begin{align}
 \affmap(X)&=b \\
 X \left(\begin{matrix} Y \\ I_{\size-\rnk}\end{matrix}\right) & = 0^{\size \times (\size-\rnk)} \label{eq:GALHDGKA}
\end{align}
\end{subequations}


Let $\inc_{\rnk,\iota}\subseteq \sym^\size_\C \times \C^{\size \times (\size-\rnk)}$ be the set of complex solutions of~\eqref{eq:GHOAISFH}.
Note that it is essentially the same incidence variety defined in the previous section (for $\iota = [r]$),
modulo notation: matrix $Y$ in eq.~\eqref{sssec:det_inc} has a larger size, but all additional entries are fixed to constants.
We define $\incproj_{\rnk,\iota}$, $\inc_{\rnk,\iota}(\R)$ and $\incproj_{\rnk,\iota}(\R)$ in the same way as before:
\begin{subequations}
\begin{align}
  \incproj_{\rnk,\iota}  := \pi_{\sym^n_\C}(\inc_{\rnk,\iota}) 
  & = \{X \in \sym^\size_\C \mymid \exists Y \in \C^{\rnk \times (\size-\rnk)} \;\mbox{s.t.}\; (X,Y) \in \inc_{\rnk,\iota}\} \\
                        & = \{X \in \sym^\size_\C \mymid \exists Y \in \C^{\rnk \times (\size-\rnk)} \;\mbox{s.t.}\; \affmap(X)=b, X \left(\begin{smallmatrix} Y \\ I_{\size-\rnk}\end{smallmatrix}\right) = 0^{\size \times (\size-\rnk)}\} 
\end{align}
\end{subequations}
\begin{equation}
\inc_{\rnk,\iota}(\R)  = (\sym^\size \times \R^{\size \times (\size-\rnk)}) \cap \inc_{\rnk,\iota} \hspace{40pt} 
\incproj_{\rnk,\iota}(\R) = \sym^\size \cap \incproj_{\rnk,\iota}=\pi_{\sym^n}(\inc_{\rnk,\iota}(\R))
\end{equation}
By construction, $P\subseteq \face \cap \{X \in \sym^\size \mymid \affmap(X)=b\} \subseteq \incproj_{\rnk,\iota}(\R)$.




\begin{lemma}\label{prop:OmegaExists}
Let $P \neq \varnothing$ and let $X^* = UZ^* U^\top\in P$ be a maximum-rank solution, with $U$ as in~\eqref{eq:Uform} and $Z^*\in \sym^\rnk_{++}$ positive definite.
Define open neighborhood $\Omega^*$ of $X^\ast$ via $\Omega^\ast=\{X\in \sym^\size\::\:||X-X^\ast||<\rho(X^\ast)\}$.
The following holds:
\begin{itemize}
\item[(a)] $P \cap \Omega^* = \incproj_{\rnk,\iota}(\R) \cap \Omega^*$.
\item[(b)] For every $X \in \incproj_{\rnk,\iota}(\R) \cap \Omega^*$ there exists a unique matrix $Y$ such that $(X,Y) \in \inc_{\rnk,\iota}$. This matrix is $Y=Y_U$.  
\end{itemize}
\end{lemma}
\begin{proof} 
  {\it Part (a)}. 
   If $X \in P$ then $\affmap(X) = b$ and $(X,Y_U) \in \inc_{\rnk,\iota}$ (since $P\subseteq\calF$), and thus $X \in \incproj_{\rnk,\iota}(\R)$.
  This shows that $P \subseteq \incproj_{\rnk,\iota}(\R)$ and hence $P \cap \Omega^* \subseteq \incproj_{\rnk,\iota}(\R) \cap \Omega^*$.

  Now suppose that $X \in \incproj_{\rnk,\iota}(\R) \cap \Omega^*$. Then there exists $Y$ such that
  $(X,Y) \in \inc_{\rnk,\iota}$, i.e. $(X,Y)$ satisfies system~\eqref{eq:GHOAISFH}.
  Eq.~\eqref{eq:GALHDGKA} implies that $X$ has at least $\size-\rnk$ zero eigenvalues. 
  Since $\text{rank}(X^\ast)=r$,
  we have $\lambda_r(X^\ast)=\rho(X^\ast)>0$.
  By~\eqref{eq:lambdaInequality}, for all $X\in\Omega^\ast$ we have $|\lambda_r(X)-\lambda_r(X^\ast)|\le || X-X^\ast || < \rho(X^\ast)=\lambda_r(X^\ast)$
  and thus $\lambda_r(X) > 0$, i.e.\ $X$ has at least $r$ strictly positive eigenvalues.
  We conclude that $X$ has exactly $n-r$ zero eigenvalues and exactly $r$ strictly positive eigenvalues,
  therefore $X\succeq 0$
  and $\text{rank}(X)=r$.
  Combined with the fact that $\affmap(X)=b$, this implies that $X \in P$. We showed that 
  converse inclusion $\incproj_{\rnk,\iota}(\R) \cap \Omega^* \subseteq P \cap \Omega^*$ is also true.
    
  {\it Part (b)}. As shown in (a), we have $X\in P\subseteq\face$.
  Therefore, $(X,Y_U)\in \inc_{\rnk,\iota}$, i.e.\ it satisfies~\eqref{eq:GHOAISFH}.
  Suppose that there exists $Y \neq Y'$ such that $(X,Y)$ and $(X,Y')$ satisfy~\eqref{eq:GHOAISFH}.
  Denote 
  $$
  C := \begin{pmatrix} Y & Y' \\ I_{\size-\rnk} & I_{\size-\rnk} \end{pmatrix},$$ 
  then
  $XC=0^{\size\times 2(\size-\rnk)}$. Clearly, we have $\text{rank}(C) \geq \size-\rnk+1$: indeed,
  the first $\size-\rnk$ columns of $C$ are linearly independent, and there exists a column among
  the remaining ones which is not in the range of the first $\size-\rnk$ columns, by the assumption
  $Y \neq Y'$. This implies that $\text{rank}(X) \leq \size-\text{rank}(C) \le \size-(\size-\rnk+1)=
  \rnk-1$, which contradicts condition $\text{rank}(X)=r$ shown in part (a).
\end{proof}

\section{Description of the method}\label{sec:descMethod}
Motivated by Lemma~\ref{prop:OmegaExists}, we propose the following hybrid method for certifying the feasibility of system~\eqref{FEASp} (see Algorithm~\ref{alg:method}). (We assume that the system is feasible, i.e.\ $P \neq \varnothing$).
Note that some steps of Algorithm~\ref{alg:method} require approximate computations. First, we need to compute approximate solution $\tilde X$ which is close to some maximum-rank solution $X^\ast$; 
Theorems~\ref{th:method} and~\ref{th:numerical} given later will characterize how close we need to be for the method to work correctly. Second, steps 2 and 4 invoke a procedure for {\bf numerically} computing a maximal set of linearly independent columns of a given matrix. Such procedure will be described later in Section~\ref{sec:rank-revealing}.

An illustration of Algorithm~\ref{alg:method} on a small numerical example is given in Section~\ref{sec:DruWoExample}.

\begin{algorithm}[h] 
{\bf 1}. Solve system~\eqref{FEASp} numerically to get solution $\tilde X\in \sym^\size$ which is close to some maximum-rank solution $X^\ast$.
\\
{\bf 2}. Find numerically a maximal set $\iota\subseteq[n]$ of linearly independent columns of $\tilde X$, let $r=|\iota|$. Assume for notational convenience that $\iota=[r]$ (this can be achieved by renaming variables). Write 
$$
\tilde X:=\begin{pmatrix} \tilde S & \tilde R^\top \\ \tilde R & \tilde W \end{pmatrix}
$$ where $\tilde S =\tilde X^\iota_\iota\in \sym^\rnk$. If $\tilde S$ is singular then terminate with failure, otherwise define 
$\tilde Y := -\tilde S^{-1}\tilde R^\top\in\mathbb R^{r\times(n-r)}$. (Note, \Cref{lemma:SymSubmatrix} implies that $\tilde X\left(\begin{smallmatrix} \tilde Y \\ I_{n-r} \end{smallmatrix}\right)\approx 0^{n\times(n-r)}$). \\
{\bf 3}. Write system~\eqref{eq:GHOAISFH} as
\begin{equation}\label{eq:sysQYXq}
\mathcal{Q}(Y)X_{\text{hvec}}=q
\end{equation}
where $X_{\text{hvec}}$ is a  vector of dimension $k=\frac{{n(n+1)}}{2}$ obtained by vectorizing the lower triangular part of $X$, $\mathcal{Q}(Y)$ is a matrix that depends linearly on $Y$, and $q$ is fixed vector. \\ 
{\bf 4}. Find numerically a maximal set $J\subseteq[k]$ of linearly independent columns of $\mathcal{Q}(\tilde Y)$. \\
{\bf 5}. Augment system~\eqref{eq:GHOAISFH} by fixing variables in $J':=[n]-J$:
\begin{subequations}\label{eq:fixed-variables}
\begin{align}
 \affmap(X)&=b \label{eq:fixed-variables:a} \\
 X \left(\begin{matrix} Y \\ I_{\size-\rnk}\end{matrix}\right) & = 0^{\size \times (\size-\rnk)} \label{eq:fixed-variables:b} \\
 (X_{\text{hvec}})_j & =(\tilde X_{\text{hvec}})_j \qquad \forall j\in J' \label{eq:fixed-variables:c}
\end{align}
\end{subequations}
Compute one point per connected component of the real algebraic variety defined by \Cref{eq:fixed-variables}
as in \cite{HNS2015c} (cf. \Cref{ssec:lagrange}). By \cite[Lemma 3.2]{HNS2015c}, the equations in 
\eqref{eq:fixed-variables:b} at positions $i,j$ 
with $i-j>r$ are implied by other equations and thus can be omitted.
\caption{Hybrid algorithm for system $\affmap(X) = b, \,\,\, X \succeq 0$.}\label{alg:method}
\end{algorithm}

\vspace{0.5cm}

The next observation justifies Step 2.
\begin{lemma}\label{lemma:SymSubmatrix}
Consider a PSD matrix $X = \left(\begin{smallmatrix} S & R^\top \\ R &  W \end{smallmatrix}\right)\in \sym^\size_+$ with $ S \in \sym^\rnk$,
and suppose that $\iota=[r]$ is a maximal set of linearly independent columns (so that the rank of $X$ is $r$).
Then $S$ is non-singular (and thus positive definite), and matrix 
 $Y=-S^{-1} R^\top$ satisfies 
 $X\left(\begin{smallmatrix} Y \\ I_{n-r} \end{smallmatrix}\right) = 0^{n\times(n-r)}$.
\end{lemma}
\begin{proof}
By hypothesis, the last $n-r$ columns of $X$ are linearly dependent on the first $r$ columns. Therefore, we can write
$$ \begin{pmatrix} R^\top \\ W \end{pmatrix} = \begin{pmatrix} S \\ R \end{pmatrix} Z $$
for some matrix $Z \in \mathbb{R}^{r \times (n-r)}$, which means $R^\top = SZ$ and $W = RZ$.  \rev{Hence, $R = Z^\top S$ and $W = Z^\top S Z$. Thus, we can write $X = \begin{pmatrix} I & Z \end{pmatrix}^\top S \begin{pmatrix} I & Z \end{pmatrix}$.  This implies that $r = \text{rank}(X) \leq \text{rank}(S)$, and the first claim holds.}

Since $S$ is non-singular, we have $Z = S^{-1}R^\top$ and $W = RS^{-1}R^\top$.  Using these equations, we can check that $SY + R^\top = 0^{r \times (n-r)}$ and $RY + W = 0^{(n-r) \times (n-r)}$, which yields the second claim.
\end{proof}


\begin{definition}
We say that sets $\iota \subseteq [n]$ and $J \subseteq [k]$ are {\em{valid}} for a maximum-rank solution $X^\ast$ if they can be outputs of steps 2 and 4, respectively, assuming that step 1 computes the vector $X^\ast$, and the computations in steps 2 and 4 are performed exactly
(i.e.\ a maximal set of linearly independent columns is computed {\em exactly} rather than numerically).
\end{definition}

\begin{theorem}\label{th:method}
Each maximum-rank solution $X^\ast$ has an open neighborhood $\tilde\Omega\subset\sym^\size$ 
with the following property.
Suppose that step 1 produces point $\tilde{X}\in\tilde\Omega$,
 and lines 2 and 4 output sets $\iota$ and $J$, respectively, which are valid for $X^\ast$. Then 
 the set of real solutions of system~\eqref{eq:fixed-variables} has a zero-dimensional component with a solution $(X, Y)$ where $X \in \sym^\size_+$.
\end{theorem}
Note, the theorem's precondition essentially requires that approximate computations output the same sets $\iota,J$ as the exact computations would. Later in Section~\ref{sec:rank-revealing} we will show that this can be achieved even if $X^\ast$ is known only approximately.
\begin{proof}
    By assumption, set $\iota$ is valid for $X^\ast$, i.e.\ $\iota$ is a maximal set of linearly independent columns in $X^\ast$ (and so $|\iota|=r$).
    Assume for notational convenience that $\iota=[r]$.
    By Lemma~\ref{lemma:SymSubmatrix}, $\iota$ is also a maximal set of linearly independent rows in $X^\ast$.
    Choose matrix $U\in\mathbb R^{n\times r}$ characterizing the face $\calF$ as in~\eqref{eq:face-range}.
    Since $\text{range}(X^\ast)\subseteq\text{range}(U)$ and $\text{rank}(X^\ast)=r$,
    set $\iota$ is a maximal set of linearly independent rows in $U$. Thus, we can choose $U$ as in~\eqref{eq:Uform}.
   
    Denote $\Omega^\ast=\{X\in\sym^\size\::\:||X-X^\ast||<\rho(X^\ast)\}$. By Lemma~\ref{prop:OmegaExists}(a,b), for any $(X,Y)$ we have 
\begin{equation}\label{eq:equivProp1}
\begin{array}{c}
    \mathcal{Q}(Y) X_{\text{hvec}} = q  \\
    X \in \Omega^\ast
\end{array}
\;\;\Longleftrightarrow\;\;
\begin{array}{c}
    \mathcal{Q}(Y_U) X_{\text{hvec}} = q  \\
    X \in \Omega^\ast \\
    Y = Y_U
\end{array}
\;\;\Longrightarrow\;\;
X\in\sym^\size_+
\end{equation}
To simplify notation, let us write 
$$
X_\text{hvec} =: \begin{pmatrix}X^{J}\\X^{J'}\end{pmatrix}
$$
and $\mathcal{Q}(Y_U)=:\mathcal{Q}=:\begin{pmatrix}Q^{J} & Q^{J'}\end{pmatrix}$.
By assumption, $J$ is a maximal set of linearly independent columns for matrix $\mathcal{Q}$.
Let $I$ be a maximal set of linearly independent rows of $\mathcal{Q}$, then $|I|=|J|=\text{rank}(\mathcal{Q})$.
The linear system $\mathcal{Q}X_{\text{hvec}}=q$ has at least one solution, namely $X_{\text{hvec}}=X^\ast_{\text{hvec}}$.
Thus, removing equations corresponding to rows $i\notin I$ will not affect the set of feasible solutions.
This implies that system $\mathcal{Q}X_{\text{hvec}}=q=\mathcal{Q}X^\ast_{\text{hvec}}$ is equivalent
to the system $$\begin{pmatrix}Q^{J}_I & Q^{J'}_I\end{pmatrix} \begin{pmatrix}X^{J}-(X^\ast)^J\\X^{J'}-(X^\ast)^{J'}\end{pmatrix}=0,$$
which is in turn equivalent to  
$X^J-(X^\ast)^J=-(\mathcal{Q}^J_I)^{-1}\mathcal{Q}^{J'}_I (X^{J'}-(X^\ast)^{J'})$
since matrix $\mathcal{Q}^J_I$ is nonsingular. Let us denote $\mathcal{R}^J_I=-(\mathcal{Q}^J_I)^{-1}\mathcal{Q}^{J'}_I$.

By adding constraint $X^{J'}=\tilde X^{J'}$ to~\eqref{eq:equivProp1}
we can now conclude the following for any $(X,Y)$:
\begin{equation}\label{eq:equivProp2}
\begin{array}{c}
    (X,Y)\mbox{ satisfies~\eqref{eq:fixed-variables}}  \\
    X \in \Omega^\ast
\end{array}
\;\;\Longleftrightarrow\;\;
\begin{array}{c}
    X^J-(X^\ast)^J=\mathcal{R}^J_I (\tilde X^{J'}-(X^\ast)^{J'})  \\
    X^{J'}=\tilde X^{J'} \\
    X \in \Omega^\ast \\
    Y = Y_U
\end{array}
\;\;\Longrightarrow\;\;
X\in\sym^\size_+
\end{equation}

Let us define 
$$
\delta_{\iota,J,I} := \frac{\rho(X^\ast)}{4\max\{||\mathcal{R}^J_I ||,1\}}.
$$
We claim that for any $\tilde X\in\sym^\size$ with $||\tilde X-X^\ast||_F<\delta_{\iota,J,I}$ the system in the middle of~\eqref{eq:equivProp2}
has exactly one feasible solution $(X,Y)$.
Indeed, we only need to verify that $X$ defined by the equalities in this system satisfies
$X\in \Omega^\ast$. This holds since $||X^J-(X^\ast)^J||\le ||\mathcal{R}^J_I||\cdot ||\tilde X^{J'}-(X^\ast)^{J'}|| < \frac 14 \rho(X^\ast) $
and $||X^{J'}-(X^\ast)^{J'}|| < \frac 14 \rho(X^\ast)$, implying $||X-X^\ast|| \le ||X-X^\ast||_F \le 2||X^{J}-(X^\ast)^{J}||+2||X^{J'}-(X^\ast)^{J'}|| <  \rho(X^\ast)$.

We can now define set $\tilde\Omega$ in Theorem~\ref{th:method} as follows:
$\tilde\Omega=\{\tilde X\in\sym^\size\::\:||\tilde X-X^\ast||_F<\delta\}$
where $\delta=\min_{\iota,J,I}\delta_{\iota,J,I}$ and the minimum is taken over (finitely many) valid choices of $\iota,J,I$.
\end{proof}


\subsection{Numerical algorithms: implementing steps 2 and 4}\label{sec:rank-revealing}
To implement Algorithm~\ref{alg:method}, we need a procedure for numerically computing a maximal set of linearly
independent columns of a given matrix. This can be done, for example, 
by a rank-revealing Gaussian elimination algorithm~\cite{Pan:00,SchorkGondzio:20}.
Such algorithm takes a matrix $A\in\mathbb R^{p\times q}$ and a tolerance value $\epsilon$ as input,
and produces an integer $r$ and a subset $J\subseteq[q]$ of size $r$
with the following guarantees:~\footnote{The algorithm actually
produces subsets $I\subseteq[p]$ and $J\subseteq[q]$ of size $r$ such that $\sigma_{r}(A_I^J)  \ge  \sigma_r(A)/c_{pq}$.
(This follows by combining Lemma 3.1 and Theorem 4.1 in~\cite{SchorkGondzio:20}.)
This implies~\eqref{eq:RankRevealing:c} since $\sigma_r(A^J)\ge \sigma_r(A^J_I)$ by a well-known property of singular values, see e.g.~\cite[Corollary 8.6.3]{golub2013matrix}.}
\begin{subequations}\label{eq:RankRevealing}
\begin{eqnarray}
\sigma_r(A) & \ge & \epsilon \label{eq:RankRevealing:a} \\
\sigma_{r+1}(A) & \le & c_{pq} \epsilon \label{eq:RankRevealing:b} \\
\sigma_{r}(A^J) & \ge & \sigma_r(A) / c_{pq}  \label{eq:RankRevealing:c}
\end{eqnarray}
\end{subequations}
where constant $c_{pq}>1$ depends polynomially on dimensions $p,q$.
Below we analyze Algorithm~\ref{alg:method} assuming
that steps 2 and 4 use the method above with tolerance values $\epsilon_1$ and $\epsilon_2$, respectively.
The next result shows that if $||\tilde X-X^\ast||$ is sufficiently small
then from $\tilde X$ it is possible to compute sets $\iota,J$ which are valid for $X^\ast$.

\begin{theorem}\label{th:numerical}
(a) Suppose that $\delta<\epsilon_1<(\rho^\ast-\delta)/c_{nn}$ where $\delta=||\tilde X-X^\ast||$ and $\rho^\ast=\rho(X^\ast)$.
Then set $\iota$ computed in step 2 is a maximal set of linearly independent columns for $X^\ast$.
Furthermore, 
\begin{equation}\label{eq:phi}
\lambda_r(\tilde S)\;\;\ge\;\; \phi(\delta)\;\;:=\;\;\frac {((\rho^\ast-\delta)/c_{nn}-\delta)^2}{n||X^\ast||} - \delta
\end{equation}
(b) Suppose the precondition in (a) holds and $\phi(\delta)>\delta$. 
Write 
$$
X^\ast=\pmtx{S^\ast &  (R^\ast)^\top \\ R^\ast &  W^\ast}
$$ where $S^\ast=(X^\ast)^\iota_\iota$, and let $Y^\ast=-(S^\ast)^{-1} (R^\ast)^\top$.
Let $||\mathcal{Q}||$ be the 2-norm of linear operator $\mathcal{Q}:\mathbb R^{r\times(n-r)}\rightarrow\mathbb R^{p\times q}$,
and denote
\begin{equation}
\psi(\delta)=||\mathcal{Q}||\cdot\left[\frac{||X^\ast||}{\phi(\delta)-\delta}+1\right]\cdot\frac{\delta}{\phi(\delta)}
\end{equation}
If $\psi(\delta)<\epsilon_2<(\rho(\mathcal{Q}(Y^\ast))-\psi(\delta))/c_{pq}$ then set $J$ computed in step 4 is a maximal set of linearly independent columns for $\mathcal{Q}(Y^\ast)$. 
\end{theorem}
Note that $\lim_{\delta\rightarrow 0}\phi(\delta)=const>0$ and $\lim_{\delta\rightarrow 0}\psi(\delta)=0$.
Thus, the preconditions in (a,b) can be satisfied for sufficiently small $\delta>0$
as long as $\rho(\mathcal{Q}(Y^\ast))>0$. (The case when $\rho(\mathcal{Q}(Y^\ast))=0$ is not interesting:
we then have $\mathcal{Q}(Y^\ast)=0$ and thus $\calA(X)=0$ for all $X$). 
We remark that linear operator $\mathcal{Q}$, function $\psi(\cdot)$ and value $\rho(\mathcal{Q}(Y^\ast))-\psi(\delta)$ depend on $\iota$, but there are only finitely many subsets $\iota\subseteq[n]$
which give maximal linearly independent columns for $X^\ast$, so we could take the minimum value over such $\iota$'s 
when formulating the final condition.

In the remainder of this section we prove Theorem~\ref{th:numerical}. We start with part (a).
Let $r^\ast=\text{rank}(X^\ast)$ be the true rank and $r$ be the rank computed by the numerical procedure in step~2.
Condition~\eqref{eq:RankRevealing:c} gives 
$$
\sigma_r(\tilde X^\iota)\ge \frac{\sigma_r(\tilde X)}{c_{nn}} \ge 
\frac{\sigma_r(X^\ast)-||\tilde X-X^\ast||}{c_{nn}} = \frac{\rho^\ast-\delta}{c_{nn}}.
$$
Therefore, $\sigma_r((X^\ast)^\iota)\ge \sigma_r(\tilde X^\iota) - ||X^\ast-\tilde X||\ge (\rho^\ast-\delta)/c_{nn} -\delta > 0$,
implying that columns in $\iota$ are linearly independent in~$X^\ast$ (and so $r\le r^\ast$).
If $r<r^\ast$ then $\sigma_{r+1}(\tilde X)\le c_{nn}\epsilon_1$ by condition~\eqref{eq:RankRevealing:b}
and $\sigma_{r+1}(\tilde X)\ge \sigma_{r^\ast}(\tilde X)\ge \sigma_{r^\ast}(X^\ast) - ||\tilde X-X^\ast||=\rho^\ast-\delta$ by~\eqref{eq:sigmaInequality},
which is a contradiction. This shows that $r=r^\ast$. 

The last claim~\eqref{eq:phi} will follow from the result below for matrix $X=\tilde X$
and value $\tau=\sigma_r(\tilde X^\iota)\ge (\rho^\ast-\delta)/c_{nn}>\delta$.

\begin{lemma}\label{lemma:PSDsubmatrix}
Consider symmetric matrices $X^\ast\in\sym^\size_+$, $X\in\sym^\size$ with $||X^\ast-X||=\delta$.
Let  $\iota\subseteq[n]$ be a subset of size $r$
with $\sigma_r(X^\iota)=\tau\ge\delta$. Then 
$$
\lambda_r(X^\iota_\iota)\ge \frac {(\tau-\delta)^2}{n||X^\ast||} - \delta.$$
\end{lemma}
Note that this lemma also implies the first part of Lemma~\ref{lemma:SymSubmatrix} (by setting $X=X^\ast$).

\begin{proof}
We will use a well-known fact that for any matrix $A\in\mathbb R^{p\times q}$ with $p\ge q$ we have $\sigma_q(A)=\min\{||Au||\::\:u\in\mathbb R^q,||u||=1\}$.

Consider vector $u\in \mathbb R^\iota$ with $||u||=1$, then $||X^\iota u|| \ge \tau$.
Let $v\in\mathbb R^n$ be the vector with $v_i=u_i$ for $i\in \iota$ and $v_i=0$ for $i\in[n]-\iota$.
We have $||Xv||=||X^\iota u||\ge \tau$, and hence
$$
||X^\ast v||\ge ||Xv||-||(X^\ast-X)v|| \ge \tau - ||X^\ast-X||\cdot 1=\tau-\delta
$$
Let us write $X^\ast=\sum_{i=1}^n \lambda_i v_i v_i^\top=\sum_{i=1}^nw_iw_i^\top$ where $||X^\ast||=\lambda_1\ge\lambda_2\ge\ldots\ge\lambda_n\ge 0$,
$v_1,\ldots,v_n$ are orthonormal vectors, and
$w_i=\sqrt{\lambda_i}v_i$ (then $||w_i||\le \sqrt{\lambda_1}$).
Denote $\alpha_i=|v^\top w_i|$, then $v^\top X^\ast v=\sum_{i=1}^n\alpha_i^2$. We can write
\begin{align*}
\tau-\delta&\le||X^\ast v||=||\sum_{i=1}^n w_i(v^\top w_i)||\le \sum_{i=1}^n \alpha_i ||w_i|| \\ 
& \le \sqrt{\lambda_1}\cdot\sum_{i=1}^n\alpha_i 
\le \sqrt{\lambda_1}\cdot\sqrt{n\sum_{i=1}^n\alpha_i^2}
= \sqrt{\lambda_1 n(v^\top X^\ast v)}
\end{align*}
where the last inequality is tight if and only if $\alpha_1=\ldots=\alpha_n$. Therefore,
$$
u^\top X^\iota_\iota u = v^\top X v = v^\top X^\ast v + v^\top (X-X^\ast) v 
\ge \frac {(\tau-\delta)^2}{n\lambda_1} - ||v||^2 \cdot ||X-X^\ast||
= \frac {(\tau-\delta)^2}{n\lambda_1} - \delta
$$
\hbox{}
\end{proof}

Next, we prove part (b) of Theorem~\ref{th:numerical}. 
We will need the following result.

\begin{theorem}[{\cite[Theorem 7.12]{Atkinson}}]\label{th:InverseStability}
Let $A, B$ be square matrices of the same size such that $A$ is nonsingular and $||A-B||\le \frac{1}{||A^{-1}||}$.
Then $B$ is also nonsingular and
$$
||B^{-1}||\le \frac{||A^{-1}||}{1-||A^{-1}||\cdot ||A-B||}
$$ 
$$
||A^{-1}-B^{-1}||\le \frac{||A^{-1}||^2\cdot ||A-B||}{1-||A^{-1}||\cdot ||A-B||}
$$
\end{theorem}

For brevity, let us denote  $S=S^\ast$ and $R=R^\ast$.
We can write 
$$
\tilde Y-Y
=S^{-1}R^\top - \tilde S^{-1}\tilde R^\top
=(S^{-1}-\tilde S^{-1}) R^\top + \tilde S^{-1}(R^\top - \tilde R^\top)
$$
From part (a) we know that $||\tilde S^{-1}||\le 1/{\phi(\delta)}$,
and so $||\tilde S^{-1}-S||\le||\tilde X-X^\ast||=\delta<\phi(\delta)\le{1}/{||\tilde S^{-1}||} $. Using Theorem~\ref{th:InverseStability}
with $A=\tilde S$, $B=S$ we get
\begin{align*}
||\tilde Y-Y^\ast||&\le ||S^{-1}-\tilde S^{-1}||\cdot|| R^\top|| + ||\tilde S^{-1}||\cdot ||R^\top - \tilde R^\top|| \\
& \le \frac{(\frac 1 {\phi(\delta)})^2\cdot \delta}{1-\frac 1 {\phi(\delta)}\cdot \delta}\cdot ||X^\ast|| + \frac{1}{\phi(\delta)}\cdot \delta
=\left[\frac{||X^\ast||}{\phi(\delta)-\delta}+1\right]\cdot\frac{\delta}{\phi(\delta)}
\end{align*}

Denote $\mathcal{Q}^\ast=\mathcal{Q}(X^\ast)$ and $\tilde{\mathcal{Q}}= \mathcal{Q}(\tilde X)$,
then $||\tilde{\mathcal{Q}}-\mathcal{Q}^\ast||\le||\mathcal{Q}||\cdot ||\tilde Y-Y^\ast||\le \psi(\delta)$.
We now conclude that $J$ is a maximal set of linearly independent columns of $\mathcal{Q}^\ast$
by the same argument as in part (a).


\subsection{Symbolic algorithms: implementing step 5}
\label{ssec:lagrange}
Step 5 of Algorithm \ref{alg:method} and \Cref{th:method}, similarly to \cite{HNS2015c}, reduce the original problem of validating the 
feasibility of a LMI to a classical question in real algebraic geometry, the computation of one point per connected component of the real 
trace $V(\R) = V \cap \R^n$ of a complex variety 
$$
V = \{x \in \C^n : f_1(x)=0, \ldots, f_c(x) = 0\}.
$$
A standard technique consists in computing the critical points of a well-chosen polynomial map
$\varphi : \C^n \to \C$ (usually of low degree) restricted to $V$, often called {\it critical point method} (CPM). 
The expected output of a CPM is a finite set intersecting every connected component of $V(\R)$.
For an overview of this problem, see e.g. \cite[\S~12.6]{BaPoRo06}.

Usual choices for $\varphi$ are linear functions: when $\varphi$ is a generic linear function, and since
$V(\R)$ has finitely many connected components, there are finitely many critical points of its restriction to $V(\R)$.
In case $V(\R)$ is non-compact, $\varphi$ might not admit critical points on the (unbounded) connected components
$C \subseteq V$ satistfying $\varphi(C) = \R$. In this case, the recursive method in \cite{SaSc03} allows to 
compute such components. However, our situation is special inasmuch as our target component is an isolated real point,
so that the mentioned recursion is not needed on Step 5 of Algorithm~\ref{alg:method}.

The CPM method for polynomials $f : C^n\to \C^c$ and polynomial map $\varphi : \C^n \to \C$ can be described as follows.
Recall that $J(f) = ({\partial f_i}/{\partial x_j})_{ij}$ is the Jacobian matrix of $f$.
Define the extended Jacobian of $(\varphi,f)$ via
$$
J(\varphi,f) := 
\begin{pmatrix}
\frac{\partial\varphi}{\partial x_1} & \cdots & \frac{\partial\varphi}{\partial x_n} \\
                                & J(f) &
\end{pmatrix}.
$$

Now consider the following Lagrange system associated with $\varphi$ and $V$:
\begin{equation}
\label{eq:lagrangianU}
\mathscr{L}_u:
\begin{cases}
f_1 = \cdots = f_c = 0 \\
z^T J(\varphi,f) = 0 \\
z^T u - 1 = 0
\end{cases}
\end{equation}
where $z=(z_0,z_1,\ldots,z_c)\in \C^{c+1}$ are new variables and $u\in C^{c+1}$ is a fixed vector.
Intuitively,  constraint $z^T u = 1$ for a generic $u$
is essentially equivalent to the constraint $z\ne 0$. 

With some abuse of notation, we denote by $\pi_x : \C^{n+c+1} \to \C^n$ the projection map sending $(x,z)$ to the first 
$n$ variables $x$. Let $V(\mathscr{L}_u)\subseteq \C^{n+c+1}$ be the set of solutions of system~\eqref{eq:lagrangianU}. 
Under certain conditions (e.g.\ those described below) the set $\pi_x(V(\mathscr{L}_u))$ will be finite.
In this case this set can be computed, for example, via the software \msolve\ based on Gr\"obner basis computations,
and represented through a so-called {\it rational univariate representation}.
This is given by a sequence
$$
(q,q_0,q_1,\ldots,q_n) \in \Q[t]^{n+2},
$$
with $q_0,q$ coprime, and such that
$$
\pi_x(V(\mathscr{L}_u)) = \left\{\left(\frac{q_1(t)}{q_0(t)},\ldots,\frac{q_n(t)}{q_0(t)}\right) \in \R^n : q(t) = 0\right\}.
$$
In other words, the coordinates of vectors in $\pi_x(V(\mathscr{L}_u))$ are represented by the evaluation of $n$ univariate rational functions
to the roots of a univariate polynomial $q$.   \vspace{0.3cm}

\myparagraph{Properties of CPM}
With some abuse of notation, let us denote 
\begin{align*}
\mathrm{Reg}(f):=\{x\in V\::\:\rank \, J(f) = c\} 
\qquad
\mathrm{Sing}(f):=\{x\in V\::\:\rank \, J(f) < c\} 
\end{align*}

Note that $\mathrm{Reg}(f)=\mathrm{Reg}(V)$ and $\mathrm{Sing}(f)=\mathrm{Sing}(V)$ (respectively the set of regular and singular points of $V = V(f)$), assuming that the following condition holds.

\begin{assumption}\label{A:CI}
$V$ is equidimensional of codimension $c$ and $I(V) = \langle f_1,\ldots,f_{c}\rangle$.
(Then $V$ is a complete intersection of dimension $d = n-c$, as defined in~\Cref{sec:prelim}).
\end{assumption}


A critical point of the restriction of $\varphi$ to $\mathrm{Reg}(f)$ is a point $x \in \mathrm{Reg}(f)$ such that the 
differential of the restriction of $\varphi$ to $\mathrm{Reg}(f)$ at $x$ is not surjective, that is an element of the 
constructible set
\begin{align*}
\mathrm{Crit}(\varphi,f) 
:=&\, {\{x \in \mathrm{Reg}(f) : \, \rank \, J(\varphi,f) < c+1\}} \\
=&\, {\{x \in \mathrm{Reg}(f) : \, \rank \, J(\varphi,f) = c\}}
\end{align*}
We also denote $\mathrm{Crit}(\varphi,V) = \{x \in \mathrm{Reg}(V) : \, \rank \, J(\varphi,f) = c\}$ under
\Cref{A:CI}.

\begin{proposition}\label{prop:lagrange}
Define
\begin{equation}\label{setsU}
\begin{aligned}
\mathrm{Crit}_u(\varphi,f) & = \{x\in \mathrm{Crit}(\varphi,f)\::\:\exists z\in \C^{c+1}\mbox{ s.t.\ }z^T J(\varphi,f)=0, z^T u = 1 \} \\
\mathrm{Sing}_u(\varphi,f) & = \{x\in \mathrm{Sing}(f)\::\:\hspace{9pt}\exists z\in \C^{c+1}\mbox{ s.t.\ }z^T J(\varphi,f)=0, z^T u = 1 \}
\end{aligned}
\end{equation}
\begin{itemize}
\item[(a)] The following holds: $\pi_x(V(\mathscr{L}_u)) = \mathrm{Crit}_u(\varphi,f) \cup \mathrm{Sing}_u(\varphi,f)$.
\item[(b)] If Assumption~\ref{A:CI} holds then for a generic linear form $\varphi \in \C[x]_1$, the set $\mathrm{Crit}(\varphi,V)$ 
is finite. 
\item[(c)] If $\mathrm{Crit}(\varphi,f)$ is finite then $\mathrm{Crit}_u(\varphi,f) = \mathrm{Crit}(\varphi,f)$
for a generic $u$. 
\item[(d)] If $\mathrm{Sing}(f)$ is finite then $\mathrm{Sing}_u(\varphi,f)=\mathrm{Sing}(f)$ for a generic $u$. 
\end{itemize}
Consequently, if Assumption~\ref{A:CI} holds and $\mathrm{Sing}(V)$ is finite then for generic $\varphi,u$ the set 
$\pi_x(V(\mathscr{L}_u))$ is finite and satisfies $\pi_x(V(\mathscr{L}_u)) = \mathrm{Crit}(\varphi,V) \cup \mathrm{Sing}(V)$.
\end{proposition}
\begin{proof}
\myparagraph{(a)} 
The following can be checked using simple linear algebra:
\begin{itemize}
\item Suppose that $x\in\mathrm{Reg}(f)$.
Then $x\in \mathrm{Crit}_u(\varphi,f)$ if and only if there exists $z\in \C^{c+1}$ s.t.\ $z^T J(\varphi,f)=0$ and $z^T u = 1$,
or equivalently s.t.\ $(x,z)$ satisfies~\eqref{eq:lagrangianU}.
\item Suppose that $x\in\mathrm{Sing}(f)$.
Then $x\in \mathrm{Sing}_u(\varphi,f)$ if and only if there exists $z\in \C^{c+1}$ s.t.\ $z^T J(\varphi,f)=0$ and $z^T u = 1$,
or equivalently s.t.\ $(x,z)$ satisfies~\eqref{eq:lagrangianU}.
\end{itemize}
These facts imply that $\pi_x(V(\mathscr{L}_u))\cap \mathrm{Reg}(f)=\mathrm{Crit}_u(\varphi,f)$
and $\pi_x(V(\mathscr{L}_u))\cap \mathrm{Sing}(f)=\mathrm{Sing}_u(\varphi,f)$, which yields the desired claim.

\myparagraph{(b)} The claim follows from~\cite[Lem.7]{BaGiHePa05}.

\myparagraph{(c, d)} 
For $\mathcal{X}\in\{\mathrm{Crit}(\varphi,f),\,\mathrm{Sing}(f)\}$ denote by $\mathcal{X}_u$ the corresponding 
set in \eqref{setsU}.
We claim that if $\mathcal{X}$ is finite then $\mathcal{X}_u=\mathcal{X}$ for a generic $u$.
Indeed, assume $\mathcal{X}=\{x_1,\ldots,x_k\}$. By definition of $\mathcal{X}$, for each $x_i\in\mathcal{X}$ there exists $z_i\in\C^{c+1}-\{0\}$ with $z_i^T \cdot J(\varphi,f)|_{x_i}=0$.
Define $\mathscr{U}_i=\{u\in\C^{c+1}\::\: z_i^T u\ne 0\}$, and let $\mathscr{U}=\mathscr{U}_1\cap\ldots\cap \mathscr{U}_k$.
Clearly, $\mathscr{U}$ is a non-empty Zariski-open set, and for each $u\in \mathscr{U}$ we have $\mathcal{X}_u=\mathcal{X}$.
\end{proof}

\begin{remark} The precondition of the last statement in Proposition~\ref{prop:lagrange} (that $\mathrm{Sing}(f)$ is finite) can be shown to hold for input system $f$ in certain circumstances.
For example, this is the case for the method of \cite{HNS2015c} (described in~\Cref{sec:prelim})
applied to the LMI problem $\affmap(X) = b, X \in \sym^\size_+$ with {\em generic} input data $(\affmap,b)$:
if $r_{\min}$ is the minimum rank then
for at least one subset $\iota\subseteq[n]$ of size $|\iota|=r_{\min}$
 system~\eqref{eq:GHOAISFH} 
will satisfy Assumption \ref{A:CI}, and the variety $V$ will be smooth (i.e.\ $\mathrm{Sing}(V)=\varnothing$),
by \cite[Thm.~2.1]{HNS2015c}. Accordingly, in this case the CPM method will solve the problem with generic $\varphi,u$,
i.e.\ it will find a rational univariate representation of at least one minimum-rank solution.

However, for non-generic instances $(\affmap,b)$ there are no such guarantees. It may happen, for example, that
$\mathrm{Sing}(f)$ is positive-dimensional, in which case 
the Zariski closure of $\pi_x(V(\mathscr{L}_u))$ is expected to be also positive-dimensional
causing the CPM method to fail. 
\end{remark}

\myparagraph{Radical ideals}
Instead of directly applying CPM to $f$, one can first compute the radical ideal of $V(f)$ and attempt to derive a minimal set of polynomials $g$ such that $\langle g \rangle=I(V(f))=\sqrt{\langle g \rangle}$. Subsequently, CPM is applied to $g$. This approach is referred to as CPM$_{\tt rad}$. 

In our experiments, we utilized \Macaulay\ software~\cite{M2} to compute the radical and then employed the command \textit{mingens} that attempts to find a smaller set of its generators. Note that it is not guaranteed to produce a minimal generating set (unless the radical ideal is homogeneous). We observed that CPM$_{\tt rad}$ was capable of solving some systems that CPM could not resolve. In the latter case, the ideal $\langle f \rangle$ was not radical, while the former was able to construct a polynomial system whose ideal was radical and complete intersection.


\section{Numerical Results}
\label{sec:numerical}
This section compares our method, denoted as ``\hybrid'', with the method in~\cite{HNS2015c} that we 
call ``\HNS''.\footnote{The method of~\cite{HNS2015c} is implemented in software \cite{spectra},
but some details are different, e.g.\ \cite{spectra} relies on FGb (deprecated) instead of the more recent software
\msolve\, for solving 0-dimensional polynomial systems. For a more fair comparison we used \msolve\, both for the new method and for the algorithm in \cite{HNS2015c}.
Note that the method in~\cite{HNS2015c} has been improved and 
adapted to the case when the feasible set $\pset$ is non-generic in \cite{henrion2021exact}, but this variant is not implemented and it can be
considered computationally more demanding, inasmuch as it would involve computation of bivariate rational representations of algebraic
curves, instead of univariate representations.
}
Their details are described below.

\myparagraph{\underline\hybrid}
For the implementation Algorithm~\ref{alg:method} requires a numerical solution $\Tilde{X}$ in the neighborhood $\tilde\Omega=\{\tilde X\in\sym^\size\::\:||\tilde X-X^\ast||_F<\delta\}$ of an exact solution $X^*$ as in \Cref{th:method}.
Since we are interested in weakly feasible SDPs, we cannot
use standard interior-point methods that are designed to work for instances
in which both primal and dual problems are strongly feasible. Fortunately, weakly feasible SDPs can
be tackled via facial reduction 
approaches~\cite{BorweinWolkowicz:81a,BorweinWolkowicz:81b,Pataki:13,WakiMuramatsu:13,PermenterSelfDual:17,ParriloPerm1,zhu2019sieve,lourencco2021solving}.
We chose to use the method in~\cite{BertiniSDP} which combines facial
reduction with a numerical algebraic geometry algorithm based on the Bertini software~\cite{BertiniBook}.
The resulting polynomial system~\eqref{eq:fixed-variables} was solved with a critical point method (CPM or CPM$_{\tt rad}$)
as described in~\Cref{ssec:lagrange}.

\myparagraph{\underline{\HNS}} This method constructs system~\eqref{eq:GHOAISFH} for every subset $\iota \subseteq [n]$
of increasing cardinality $r = |\iota|$, ranging from $0$ to $n-1$, and applies the critical point method to each such variety.
It stops if it finds a PSD matrix (with entries in a rational univariate representation).
By~\Cref{th:HNS2}, we can expect to obtain a minimal-rank solution in $P$ for at least one $\iota$
of size $|\iota|=r_{\min}$.
Note that in~\cite{HNS2015c} only CPM was used, which was shown to be
sufficient for generic instances $(\affmap,b)$. Since we deal with non-generic LMI instances,
we tested both CPM and CPM$_{\tt rad}$.

The Lagrange system~\eqref{eq:lagrangianU} was solved using library \msolve \, \cite{berthomieu2021msolve}.
We say that this computation {\em succeeds}
if the projection to variables $X$ is zero-dimensional, and at least one
of the solutions (represented in a rational univariate representation) is a PSD matrix.
Otherwise the computation {\em fails}. 

\myparagraph{Instances}
We applied the methods to various instances of weakly feasible SDPs extracted from the existing literature;
their description can be found in \url{https://git.ista.ac.at/jzapata/hybrid-method}.
{Note that these instances are quite sparse; to test how the methods cope with non-sparse instances,}
for each SDP linear map {$\affmap(X) = (\la A_{1},
X\ra_F, \ldots, \la A_{n}, X\ra_F)$} we created two input instances
denoted as {\em clean} (I) and {\em rotated} (T).
The latter uses the map 
$$
X \mapsto \affmap(T X T^\top) = \left(\la T^\top A_1 T,X\ra_F, \ldots, \la T^\top A_n T, X\ra_F \right)
$$ for some random matrix $T\in GL(n,\mathbb R)$. This transformation does not affect feasibility of the system, but it allows us to generate more weakly feasible SDPs that are less sparse and may become more challenging for numerical methods (\cite[Section 6]{Pataki2020}). 

\myparagraph{Results}
The outcomes of our experiments are shown in Table~\ref{tab:compResults}.
Values $r_{\min}$ / $r_{\max}$ given in the second column are the ranks of solutions found
by \HNS\ / \hybrid\ (recall that these methods search for minimum- and maximum-rank solutions, respectively).
The results are given in the format ``$\,t\, /\, t_{\tt rad}$'' where $t$ is the runtime of CPM and $t_{\tt rad}$ is
the runtime of CPM$_{\tt rad}$ (in seconds). For \hybrid\ we report additionally the time for computing approximate solution;
it is given in square brackets. Note that this time is included in both $t$ and $t_{\tt rad}$.
If the method fails then the corresponding runtime is crossed;
for example, entry ``\,$\xcancel{5.8} \,/\, 3.1$'' in the first row for \HNS\ means that CPM failed for all considered polynomial systems (corresponding to different subsets $\iota\subseteq[n])$ while CPM$_{\tt rad}$ succeeded for at least one system.

In cases where the method reaches the predefined time limit of 10 minutes, the computation is terminated, and the respective output is labeled with $\infty$.

\myparagraph{Discussion of results} 
\hybrid\ was able to solve most of our instances both with CPM and CPM$_{\tt rad}$.
In contrast, \HNS\ failed more often, especially when CPM was used.
Interestingly, in several cases \HNS\ with CPM$_{\tt rad}$ solved
the clean version of the problem but failed on the rotated version. We conjecture that 
in these cases \Macaulay\ did not find {\em minimal} generators of the radical ideal.

We can conclude that in many cases polynomial systems constructed by \hybrid\ appear to be easier to solve
compared to systems in \HNS, and also \hybrid\ is usually faster than \HNS.
(Note that the latter needs to solve many more systems).
However, there are several exceptions: \HNS\ with CPM$_{\tt rad}$ could solve
clean versions of the last four rows in Table~\ref{tab:compResults}, whereas \hybrid\ failed or did not terminate on these instances.

In Section~\ref{sec:DruWoExample} we expand the first example from Table~\ref{tab:compResults} (the clean version of DruWo2017, cf. \Cref{appA}).
We show, in particular, that in the failed cases varieties $\mathrm{Sing}(f)$ for given polynomial systems $f$ are positive-dimensional.

We also confirm that system~\eqref{eq:fixed-variables} after fixing variables
in \hybrid\ for this example indeed has a zero-dimensional real component, as predicted by \Cref{th:method}.
We observe that the system also has two other real components, which are positive-dimensional.
After applying CPM, the system becomes 0-dimensional, and so \hybrid\ with CPM succeeds.

\begin{table}[!ht]
\centering
\footnotesize
\begin{tabular}{|l|c||c|c||c|c|}
\hline
\multirow{2}{*}{SDP} & \multirow{2}{*}{$n/r_{\min}/r_{\max}$} & \multicolumn{2}{c||}{clean instances (I)}&\multicolumn{2}{c|}{rotated instances (T)} \\ 
\cline{3-6}
&  &   \HNS  & \hybrid  & \HNS & \hybrid \\ \hline
DruWo2017-2.3.2P     & 3/1/2 & \xcancel{5.8} / 3.1 & 1.1/1.6  [0.5]       & \xcancel{4.6} / \xcancel{12.7} & 1.2 / 1.6   [0.6] \\
Gupta2013-12.3P      & 3/1/2 & \xcancel{5.9} / 3.1 & 1.7/1.9  [1.0]       & \xcancel{4.8} / \xcancel{12.4} & 1.7 / 2.0   [1.1] \\
Hauenstein2.6P       & 3/1/2 & \xcancel{5.8} / 3.1 & 1.5/1.9  [0.8]       & \xcancel{4.4} / \xcancel{12.8} & 2.0 / 1.9   [1.3] \\
Helmberg2000-2.2.1P  & 3/1/2 & 2.6 / 2.8          & 1.7/1.9  [1.2]       & 2.1 / 3.4 & 1.8 / 1.8   [1.3] \\
LauVall2020-2.5.1P   & 2/1/1 & 1.6 / 1.8          & 1.1/1.3  [0.7]       & 1.2 / 2.0 & 1.1 / 1.5   [0.7] \\
LauVall2020-2.5.2P   & 3/1/2 & \xcancel{6.2} / 3.1 & 1.6/1.9  [1.0]       & \xcancel{4.4} / \xcancel{12.5} & 1.7 / 2.4   [1.0] \\
Pataki2017-4P        & 3/1/2 & \xcancel{5.9} / 3.0 & 1.5/1.9  [0.9]       & \xcancel{4.5} / \xcancel{12.6} & 1.6 / 2.0   [0.9] \\
deKlerk2002-2.1P     & 2/1/1 & 1.4 / 1.8          & 1.1/1.6  [0.7]       & 1.5 / 1.9 & 1.1 / 1.3   [0.7] \\
DruWo2017-2.3.2D     & 3/1/2 & 2.4 / 2.8          & 1.7/2.0  [1.2]       & 2.6 / 6.2 & 2.0 / 2.1   [1.3] \\
Gupta2013-12.3D      & 3/1/2 & 2.4 / 2.8          & 1.6/2.0  [1.1]       & 2.2 / 3.2 & 1.8 / 2.0   [1.3] \\
HNS2020-4.1D         & 4/2/2 & 8.7 / 9.7          & 2.4/2.9  [1.7]       & 7.0 / 17.6 & 2.9 / 3.1   [2.0] \\
Hauenstein2.6D       & 3/1/2 & 2.5 / 2.8          & 1.8/1.5  [1.2]       & 2.0 / 6.2 & 1.9 / 2.2   [1.3] \\
Helmberg2000-2.2.1D  & 3/1/2 & \xcancel{5.9} / 3.0 & 1.6/1.8  [1.0]       & \xcancel{4.1} / 12.3 & 1.8 / 2.1   [1.1] \\
Pataki2017-4D        & 3/1/2 & 2.4 / 2.9          & 1.7/1.9  [1.2]       & 2.0 / 3.4 & 2.5 / 1.8   [1.6] \\
Permenter2018-4.3.1D & 5/0/1 & 11.5 / 1.0         & 5.5/8.1  [4.5]       & 1.0 / 1.5 & 6.6 / 6.7   [5.6] \\
Permenter2018-4.3.2D & 4/2/2 & 1.0 / 1.5          & 2.3/2.7  [1.6]       & 9.1 / 129.6 & 2.6 / 3.1   [1.5] \\
PatakiCleanDim4P     & 4/1/- & \xcancel{49.5} / 4.4 & \xcancel{2.3}/\xcancel{3.3}  [1.3] & \xcancel{77.3} / \xcancel{133.4} & \xcancel{27.2} / $\infty$ [2.6]\\
PatakiCleanDim5P     & 5/1/- & $\infty$ / 6.9    & \xcancel{154.6}/\xcancel{519.2}  [3.0] & $\infty$ / $\infty$ & $\infty$ / $\infty$ [9.5]  \\
PatakiCleanDim6P     & 6/1/- & $\infty$ / 12.4   & $\infty$/$\infty$  [11.5] & $\infty$ / $\infty$ & $\infty$ / $\infty$ [68.4]  \\
HeNaSa2016-6.2P      & 6/2/- & 33.1 / 36.3        & $\infty$/$\infty$  [9.3] & $\infty$ / $\infty$ & $\infty$ / $\infty$ [10.5] \\
\hline
\end{tabular}
\caption{Comparison between \hybrid\ and \HNS\ for clean and rotated instances.}
\label{tab:compResults}
\end{table}


\section{Numerical example}\label{sec:DruWoExample}

In this section we illustrate how \hybrid\ and \HNS\ work on (the clean version of) DruWo2017.
(Recall that on this example \HNS\ with CPM fails, and all other methods succeed.)
Our conclusions are as follows:

\begin{itemize}
    \item \underline{\hybrid.} The resulting polynomial system $F$ has $\mathrm{Sing} (F)=\{(X^\ast,Y^\ast)\}$ where $(X^\ast,Y^\ast)$ is the desired solution,
    and CPM succeeds. We write down the rational univariate representation of the solution obtained with \msolve.

    \item \underline{\HNS.} In this case there are two polynomial systems to consider, $F_1$ and $F_2$.
    Sets $\mathrm{Sing} (F_1)$ and $\mathrm{Sing} (F_2)$ are both positive-dimensional, which causes CPM to fail for $F_1$ and $F_2$.
    Furthermore, the ideals $\langle F_1\rangle$ and $\langle F_2\rangle$ are not radical.
    We then compute generators of  $\sqrt{\langle F_1\rangle}$ and $\sqrt{\langle F_2\rangle}$ with \Macaulay.
    These generators turn out to be non-minimal in the first case, and minimal in the second case.
    CPM$_{\tt rad}$ fails on $F_1$ but succeeds on $F_2$.
\end{itemize}

The example comes from~\cite[Example 2.3.2]{drusvyatskiy2017many}, and can be written as follows (over matrices $X\in\sym^4$):
\begin{subequations}
\begin{align}
X_{33} &=0 \\
2 X_{13} + X_{22} &=1 \\
X  & \succeq 0
\end{align}
\end{subequations}
Its solution set $P$ is
\begin{equation*}
P=\left\{ 
\begin{pmatrix}
x & y & 0 \\
y & 1 & 0 \\
0 & 0 & 0
\end{pmatrix}
\::\:
\begin{pmatrix}
x & y \\
y & 1
\end{pmatrix}
\succeq 0
\right\}
\hspace{30pt}
\end{equation*}

\subsection{\hybrid\ method} The minimal face $\calF$ containing $P$ is given by
$$
\calF=
\left\{ 
\begin{pmatrix}
x & y & 0 \\
y & z & 0 \\
0 & 0 & 0
\end{pmatrix}
\::\:
\begin{pmatrix}
x & y \\
y & z
\end{pmatrix}
\succeq 0
\right\}
$$
This face is described by matrix
$$
U=
\begin{pmatrix}
1 & 0 \\
0 & 1 \\
0 & 0 
\end{pmatrix}
$$ 
as in~\eqref{eq:face-range},\eqref{eq:Uform}
(with $\iota=\{1,2\}$),
which corresponds to matrix
$Y_U=
\left(\begin{smallmatrix}
 0 \\
 0 
\end{smallmatrix}\right)$.
Suppose that in step 1 we found approximate solution $\tilde X$ with entries $\tilde X_{11}=1$ and $\tilde X_{12}=0$
(and other entries are close to the unique true values). Also suppose that step 2 correctly finds $\iota=\{1,2\}$,
which would give $\tilde Y \approx
\left(
\begin{smallmatrix}
0  \\
0  
\end{smallmatrix}
\right)$. 
If $\tilde X_{ij}\ne X^\ast_{ij}$ for all $i,j$ then the only
way to get a system that contains a 0-dimensional real component with a correct solution
is to fix variables $X_{11}=\tilde X_{11}$, $X_{12}=\tilde X_{12}$. Suppose that
step 4 correctly identifies these variables, then in step 5 we obtain the system
\begin{equation}\label{eq:PatakiFixed}
\begin{array}{rcl}
 X_{33}  &=&0 \\
2 X_{13} + X_{22} &=&1 \\
X_{11} & = &1 \\
X_{12} & = &0
\end{array}
\qquad
\begin{pmatrix}
X_{11}  & X_{12} & X_{13}  \\
X_{12} & X_{22} & X_{23}  \\
X_{13} & X_{23} & X_{33}  
\end{pmatrix}
\begin{pmatrix}
Y_1  \\
Y_2  \\
1
\end{pmatrix} =
0^{3 \times 2}
\end{equation}
Making substitutions $X_{11}=1$, $X_{12}=0$, $X_{33}=0$, $X_{13}=\frac 12 - \frac 12 X_{22}$ and eliminating these variables yields the system
\begin{eqnarray}\label{eq:DruWo}
F:\left\{
	\begin{array}{rcl}
	Y_{1} - \frac{X_{22}}{2} + \frac{1}{2} &= 0 \\
   	X_{23} + X_{22}  Y_{2} &= 0 \\
   	X_{23}  Y_{2} - Y_1 \left(\frac{X_{22}}{2} - \frac{1}{2}\right) &= 0
   	\end{array}
\right.
\end{eqnarray}
Its set of solutions is
$$
V(F)=\left\{(X_{22},X_{33},Y_1,Y_2)=
\left( 2t + 1,
 -2\sqrt{-\tfrac{1}{2t + 1}}t^{2} \pm \sqrt{-\tfrac{1}{2t + 1}}t,
 t,
 \sqrt{-\tfrac{1}{2t + 1}}t
\right)\::\:
t \in \C\setminus\left\{-\tfrac{1}{2}\right\}
\right\}
$$
The solutions are real for $t\in\{0\}\cup(-\infty,-\tfrac 12)$.
Accordingly, the real variety $V(F)(\R)$ contains a zero-dimensional component with
the correct solution $(1,0,0,0)$ (as predicted by~\Cref{th:method})
but also two one-dimensional components corresponding to $t\in(-\infty,-\tfrac 12)$
with ``plus'' and ``minus'' signs.

To compute $\mathrm{Sing} (F)$, we need to evaluate the rank of the Jacobian matrix
%
\[
J(F)=\begin{pmatrix}
    -\frac{1}{2} & 0 & 1 & 0 \\
    Y_{2} & 1 & 0 & X_{22} \\
    -\frac{Y_{1}}{2} & Y_{2} & \frac{1}{2} - \frac{X_{22}}{2} & X_{23}
\end{pmatrix}
\]
at points in $V(F)$.
This rank equals 2 at the point corresponding to $t=0$, and is 3 at all other points.
Hence, we have
\[\mathrm{Sing} (F) = 
\{
\left\{(X_{22}, X_{23}, Y_1, Y_2)=(1,0,0,0)\right\}
\]

Next, we formed the Lagrangian system $\mathscr{L}_u$ as in~\eqref{eq:lagrangianU} for a randomly chosen linear map $\varphi$ and vector $u\in\Z^4$,
and solved it using \msolve. It returned the following rational parametrization encoding its zero-dimensional solution set:
\[
\pi_{{X_{22},X_{23}}}( V(\mathscr{L}_u)) =
\left\{ 
\left( 
 \frac{q_{1}(t)}{q'(t)}, \frac{q_{2}(t)}{q'(t)}
\right) : q(t)=0
\right\}
\]
where $\pi_{{X_{22},X_{23}}}$ is the projection onto variables $X_{22},X_{23}$ and
\begin{align*}
    q_1(t)    &= -129 t^{4}-240 t^{3}-76 t^{2}-16 t-4,\\
    q_2(t)  &= t \left(-135 t^{4}+880 t^{3}+60 t^{2}+16 t+4\right),  \\
    q(t)    &= 27t^5 - 220t^4 - 20t^3 - 8t^2 - 4t. 
\end{align*}
The floating-point representation (up to 5 decimal places) is
\[
\pi_{{X_{22},X_{23}}}(V(\mathscr{L}_u)) \approx
\left\{
\begin{array}{l}
(-10.18376, 17.84481), \\
(-0.08682, -0.16012), \\
(1, 0), \\
(0.04640 - 0.13399i, 0.08358 + 0.16089i), \\
(0.04640 + 0.13399i, 0.08358 - 0.16089i)
\end{array}
\right\}
\]

\subsection{\HNS\ method} Since \HNS\ searches for a minimum rank solution, we need to use $|\iota|=1$. Consequently, three possibilities arise for the kernel matrix:
 \[
\begin{pmatrix}
Y_{11} & Y_{12} \\
1 & 0 \\
0 & 1 \\
\end{pmatrix},
\begin{pmatrix}
1 & 0 \\
Y_{21} & Y_{22} \\
0 & 1 \\
\end{pmatrix},
\begin{pmatrix}
1 & 0 \\
0 & 1 \\
Y_{31} & Y_{32} \\
\end{pmatrix}
 \]
It can be checked that for the third matrix $Y$ system $XY=0$ does not have solutions with $X\in P$.
Below we analyze the two remaining systems. 

\begin{enumerate}
    \item[(1)] Making substitutions $X_{33}=0$, $X_{13}=\frac 12 - \frac 12 X_{22}$ and eliminating these variables yields
\end{enumerate}
 \[
 \begin{pmatrix}
X_{11} & X_{12} & \frac{1}{2} - \frac{X_{22}}{2} \\
X_{12} & X_{22} & X_{23} \\
\frac{1}{2} - \frac{X_{22}}{2} & X_{23} & 0 \\
\end{pmatrix}
\begin{pmatrix}
Y_{11} & Y_{12} \\
1 & 0 \\
0 & 1 \\
\end{pmatrix}
=0^{3\times 2}
\quad\Longleftrightarrow\quad
F_1:\left\{
\begin{array}{rcl}
X_{12} + X_{11}Y_{11} &=& 0 \\
X_{22} + X_{12}Y_{11} &=& 0 \\
X_{11} Y_{12} - \frac{X_{22}}{2} + \frac{1}{2} &=& 0 \\
X_{23} + X_{12}Y_{12} &=& 0 \\
-Y_{12}\left(\frac{X_{22}}{2} - \frac{1}{2}\right) &=& 0 
\end{array}
\right.
\]
Recall that the equation at position (3,1) on the LHS is implied by other
equations (see step 5 of Algorithm~\ref{alg:method}), and
accordingly we removed it on the RHS.
This polynomial system has the following parametric description for its set of solutions:
$$
V(F_1)=\left\{
(X_{11},X_{12},X_{22},X_{23},Y_{11},Y_{12})=
\left(
 \tfrac{1}{t^{2}},
 -\tfrac{1}{t},
 1,
 0,
 t,
 0
 \right)\::\:t \in \mathbb{C}\setminus \left\{0\right\}
\right\}
$$
If we compute the Jacobian and replace the parametric description we can see the matrix has rank 4:
$$
J(F_1)=
\begin{pmatrix}
Y_{11} & 1 & 0 & 0 		& X_{11} & 0 \\
0 & Y_{11} & 1 & 0		& X_{12} & 0 \\
Y_{12} & 0 & -\tfrac 12 & 0 	& 0 & X_{11} \\
0 & Y_{12} & 0 & 1 			& 0 & X_{12} \\
0 & 0 & -\frac{Y_{12}}{2} & 0 	& 0 & \tfrac 12 - \tfrac{X_{22}}2 
\end{pmatrix}=
\begin{pmatrix}
t & 1 & 0 & 0 & \frac{1}{t^2} & 0 \\
0 & t & 1 & 0 & -\frac{1}{t} & 0 \\
0 & 0 & -\frac{1}{2} & 0 & 0 & \frac{1}{t^2} \\
0 & 0 & 0 & 1 & 0 & -\tfrac 1t \\
0 & 0 & 0 & 0 & 0 & 0 
\end{pmatrix}
$$

Moreover, we can see the ideal $I=\langle F_1 \rangle$ is not radical as $(X_{22}-1)^2 \in I$ but $X_{22}-1 \not\in I$. Using \Macaulay\ we can compute its radical ideal and obtain
\[
\sqrt{I}=\langle 
X_{23},
Y_{12},
X_{22} - 1,
X_{12}^2 - X_{11},
X_{12}Y_{11} + 1,
X_{12} + X_{11}Y_{11}
\rangle
\]

but in this case the set of generators is not minimal since 
\[
X_{12} + X_{11}Y_{11}= X_{12}(X_{12}Y_{11} + 1)-Y_{11}(X_{12}^2 - X_{11})
\]

\begin{enumerate}
    \item[(2)] Similarly, for the second system we have 
\end{enumerate}
 \[
 \begin{pmatrix}
X_{11} & X_{12} & \frac{1}{2} - \frac{X_{22}}{2} \\
X_{12} & X_{22} & X_{23} \\
\frac{1}{2} - \frac{X_{22}}{2} & X_{23} & 0 \\
\end{pmatrix}
\begin{pmatrix}
1 & 0 \\
Y_{21} & Y_{22} \\
0 & 1 \\
\end{pmatrix}
=0^{3\times 2}
\quad\Longleftrightarrow\quad
F_2:
\left\{
\begin{array}{rcl}
X_{11} + X_{12}Y_{21} &=& 0\\
X_{12} + X_{22}Y_{21} &=& 0\\
X_{23}Y_{21} - \frac{X_{22}}{2} + \frac{1}{2} &=& 0\\
X_{23} + X_{22}Y_{22} &=& 0\\
X_{23}Y_{22}  &=& 0
\end{array}
\right.
\]

Again, the ideal $I=\langle F_2 \rangle$ is not radical but when we obtain its radical in \Macaulay\ we get a minimal set of generators:
\[
\sqrt{I}=\langle 
X_{23},
Y_{22},
X_{12} + Y_{21},
X_{22} - 1,
X_{11} - Y_{21}^2
\rangle
\]

To summarize, CPM fails for $F_1,F_2$ while CPM$_{\tt rad}$ fails for $F_1$ but succeeds for $F_2$.

\begin{remark}
As Table~\ref{tab:compResults} shows, \HNS\ with both CPM and CPM$_{\tt rad}$ fails for the rotated version of DruWo2017.
The resulting polynomial systems were too big to analyze by hand; we conjecture
that in these cases \Macaulay\ also failed to find {\em minimal} generators of radical ideals,
as for system $F_1$.
\end{remark}


\section{Conclusions and future work}
We presented a hybrid method for certifying weakly feasible SDP problems
that uses an approximate numerical SDP solver and an exact solver for polynomial systems over reals.
Our numerical results indicate that the hybrid method can outperform a pure exact algorithm from~\cite{HNS2015c}
when given a good approximate solution.

In our current experiments scalability was limited both by the numerical SDP solver that we employed
(which was the facial reduction algorithm~\cite{BertiniSDP} based on Bertini~\cite{BertiniBook}),
and by the exact solver for polynomial systems (\msolve).
We conjecture that our approach can handle larger instances if the method in~\cite{BertiniSDP} is replaced with an alternative facial reduction algorithm,
and an exact solver for polynomial systems is replaced with a numerical solver; this is left as a future work.
Potential candidates for the latter could be algorithms from numerical algebraic geometry such as Bertini.
Using such solver could address the issue discussed in the previous sections:
if a system of polynomial equations has a zero-dimensional component and a positive-dimensional component,
a numerical solver would focus on the desired zero-dimensional component if the previous step found
a good approximation.


\appendix
\section{Description of SDP Problems}
\label{appA}

In the following, we present a list of weakly feasible semidefinite programs used for the numerical results shown in Table~\ref{tab:compResults}.

Each SDP is accompanied by a reference from which it was taken, along with a description in the standard primal form as follows:
\begin{equation*}
\label{SDPpp}
\begin{array}{rcll}
& \min_{X \in \sym^\size_+} & \left\{ \langle C, X \rangle_{F} : \affmap(X) = b\right\}
\end{array}
\end{equation*}
with $X=(x_{ij}) \in \sym^n$.

\subsubsection*{DruWo2017-2.3.2PStd~\cite{drusvyatskiy2017many}} 
\[
\min_{X \in \mathbb{S}_+^{3}} \{ x_{22} : 2x_{13} + x_{22} = 1, \; x_{33} = 0 \}
\]

\subsubsection*{Gupta2013-12.3PStd~\cite{Gupta2013}}
\[
\min_{X \in \mathbb{S}_+^{3}} \{ x_{33} : 1 - x_{33} - 2x_{12} = 0, \; x_{22} = 0 \}
\]

\subsubsection*{Hauenstein2.6PStd~\cite{BertiniSDP}}
\[
\min_{X \in \mathbb{S}_+^{3}} \{ x_{11} : x_{11} + 2x_{23} = 2, \; x_{22} = 0 \}
\]

\subsubsection*{Helmberg2000-2.2.1PStd~\cite{Helmberg2002}}
\[
\min_{X \in \mathbb{S}_+^{3}} \{ x_{12} : x_{33} - x_{12} = 1, \; x_{11} = 0, \; x_{13} = 0, \; x_{23} = 0 \}
\]

\subsubsection*{LauVall2020-2.5.1PStd~\cite{Laurent2020}}
\[
\min_{X \in \mathbb{S}_+^{3}} \{ -2x_{12} : x_{11} = 1, \; x_{22} = 0 \}
\]

\subsubsection*{LauVall2020-2.5.2PStd~\cite{Laurent2020}}
\[
\min_{X \in \mathbb{S}_+^{3}} \{ -x_{11} - x_{22} : x_{11} = 0, \; 2x_{13} + x_{22} = 1 \}
\]

\subsubsection*{Pataki2017-4PStd~\cite{Pataki2018}}
\[
\min_{X \in \mathbb{S}_+^{3}} \{ x_{11} + x_{22} : x_{11} = 0, \; 2x_{13} + x_{22} = 1 \}
\]

\subsubsection*{deKlerk2002-2.1PStd~\cite{DeKlerk2002}}
\[
\min_{X \in \mathbb{S}_+^{3}} \{ 2x_{12} + x_{22} : x_{11} = 0, \; x_{22} - 1 = 0 \}
\]

\subsubsection*{DruWo2017-2.3.2DStd~\cite{drusvyatskiy2017many}}
\[
\min_{X \in \mathbb{S}_+^{3}} \{ x_{22} : x_{11} = 0, \; x_{12} = 0, \; x_{22} - x_{13} - 1 = 0, \; x_{23} = 0 \}
\]

\subsubsection*{Gupta2013-12.3DStd~\cite{Gupta2013}}
\[
\min_{X \in \mathbb{S}_+^{3}} \{ x_{33} : x_{11} = 0, \; 2x_{13} = 0, \; 2x_{23} = 0, \; x_{33} - x_{12} - 1 = 0 \}
\]

\subsubsection*{HNS2020-4.1DStd~\cite{HNS2015b}}
\begin{align*}
&\min_{X \in \mathbb{S}_+^{4}} \{ x_{11} + 2x_{22} + 4x_{34} : x_{11} = 1, \; x_{13} = 0, \; x_{14} = 0, \; x_{22} = 2, \\ 
&\quad x_{23} = 0, \; x_{24} = 0, \; x_{33} - 2x_{12} = 0, \; 2x_{34} = 4, \; x_{44} - x_{12} = 0 \}
\end{align*}

\subsubsection*{Hauenstein2.6DStd~\cite{BertiniSDP}}
\[
\min_{X \in \mathbb{S}_+^{3}} \{ x_{11} : 2x_{12} = 0, \; 2x_{13} = 0, \; 2x_{23} - 2x_{11} + 2 = 0, \; x_{33} = 0 \}
\]

\subsubsection*{Helmberg2000-2.2.1DStd~\cite{Helmberg2002}}
\[
\min_{X \in \mathbb{S}_+^{3}} \{ x_{12} : x_{22} = 0, \; 2x_{12} + x_{33} - 1 = 0 \}
\]

\subsubsection*{Pataki2017-4DStd~\cite{Pataki2018}}
\[
\min_{X \in \mathbb{S}_+^{3}} \{ x_{11} + x_{22} : 2x_{12} = 0, \; x_{22} - x_{13} - 1 = 0, \; 2x_{23} = 0, \; x_{33} = 0 \}
\]

\subsubsection*{Permenter2018-4.3.2DStd~\cite{ParriloPerm1}}
\begin{align*}
&\min_{X \in \mathbb{S}_+^{4}} \{ x_{11} - x_{22} - x_{33} + x_{44} : x_{11} = 1, \; x_{13} = 0, \; x_{14} + x_{23} = 0, \\
&\quad x_{24} = 0, \; 2x_{12} + x_{33} + 1 = 0, \; x_{22} + 2x_{34} = -1, \; x_{44} = 1 \}
\end{align*}

\subsubsection*{Permenter2018-4.3.1DStd~\cite{ParriloPerm1}}
\begin{align*}
&\min_{X \in \mathbb{S}_+^{5}} \{ 0 : x_{12} = 0, \; x_{13} = 0, \; x_{14} = 0, \; x_{15} = 0, \; x_{11} + x_{22} = 0, \; x_{24} = 0, \\
&\quad x_{25} = 0, \; x_{34} = 0, \; x_{35} = 0, \; x_{33} - x_{23} + x_{44} = 0, \; x_{45} = 0 \}
\end{align*}

\subsubsection*{PatakiCleanDim4PStd~\cite{Pataki2020}}
\[
\min_{X \in \mathbb{S}_+^{4}} \{ x_{11} + x_{22} + x_{33} : x_{11} = 0, \; 2x_{14} + x_{22} = 0, \; 2x_{24} + x_{33} = 10 \}
\]

\subsubsection*{PatakiCleanDim5PStd~\cite{Pataki2020}}
\[
\min_{X \in \mathbb{S}_+^{5}} \{ x_{11} + x_{22} + x_{33} + x_{44} : x_{11} = 0, \; 2x_{15} + x_{22} = 0, \; 2x_{25} + x_{33} = 0, \; 2x_{35} + x_{44} = 10 \}
\]

\subsubsection*{PatakiCleanDim6PStd~\cite{Pataki2020}}
\begin{align*}
&\min_{X \in \mathbb{S}_+^{6}} \{ x_{11} + x_{22} + x_{33} + x_{44} + x_{55} : \\
&\quad x_{11} = 0, \; 2x_{16} + x_{22} = 0, \; 2x_{26} + x_{33} = 0, \; 2x_{36} + x_{44} = 0, \; 2x_{46} + x_{55} = 10 \}
\end{align*}

\subsubsection*{HeNaSa2016-6.2PStd~\cite{henrion2021exact}}
\begin{align*}
&\min_{X \in \mathbb{S}_+^{6}} \{ x_{11} - 3x_{15} + x_{23} - 4x_{25} + x_{33} + 2x_{44} + x_{46} + x_{56} + x_{66} : \\
&\quad x_{11} = 1, \; x_{12} = 0, \; x_{14} = 0, \; 2x_{13} + x_{22} = 0, \; 2x_{23} = 1, \; 2x_{15} + 2x_{24} = -3, \\
&\quad x_{33} = 1, \; 2x_{25} + 2x_{34} = -4, \; x_{35} = 0, \; 2x_{16} + x_{44} = 2, \; x_{26} + x_{45} = 0, \\
&\quad 2x_{46} = 1, \; 2x_{36} + x_{55} = 0, \; 2x_{56} = 1, \; x_{66} = 1 \}
\end{align*}

\bibliographystyle{siamplain}
\bibliography{biblio}

\def\cfac#1{\ifmmode\setbox7\hbox{$\accent"5E#1$}\else
  \setbox7\hbox{\accent"5E#1}\penalty 10000\relax\fi\raise 1\ht7
  \hbox{\lower1.15ex\hbox to 1\wd7{\hss\accent"13\hss}}\penalty 10000
  \hskip-1\wd7\penalty 10000\box7}
\begin{thebibliography}{10}

\bibitem{Gupta2013}
{\sc G.~Anupam and R.~O'Donnell}, {\em Lecture notes for cmu's course on linear
  programming \& semidefinite programming}, 2013.
\newblock Lecture notes for Carnegie Mellon University's course.

\bibitem{Atkinson}
{\sc K.~Atkinson}, {\em An introduction to numerical analysis}, Wiley, 2nd~ed.,
  1991.

\bibitem{BaGiHePa05}
{\sc B.~Bank, M.~Giusti, J.~Heintz, and L.-M. Pardo}, {\em Generalized polar
  varieties: geometry and algorithms}, Journal of Complexity, 21 (2005),
  pp.~377--412.

\bibitem{BaPoRo06}
{\sc S.~Basu, R.~Pollack, and M.-F. Roy}, {\em Algorithms in real algebraic
  geometry}, vol.~10 of Algorithms and Computation in Mathematics,
  Springer-Verlag, second~ed., 2006.

\bibitem{BertiniBook}
{\sc D.~Bates, J.~Hauenstein, A.~Sommese, and C.~Wampler}, {\em Numerically
  Solving Polynomial Systems with Bertini}, SIAM, Philadelphia, PA, USA, 2013.

\bibitem{berthomieu2021msolve}
{\sc J.~Berthomieu, C.~Eder, and M.~Safey El~Din}, {\em {\sc msolve}: A library
  for solving polynomial systems}, in Proceedings of the 2021 International
  Symposium on Symbolic and Algebraic Computation, 2021, pp.~51--58.

\bibitem{BorweinWolkowicz:81a}
{\sc J.~Borwein and H.~Wolkowicz}, {\em Facial reduction for a cone-convex
  programming problem}, J. Aust. Math. Soc. (Ser. A), 30 (1981), pp.~369--380.

\bibitem{BorweinWolkowicz:81b}
{\sc J.~Borwein and H.~Wolkowicz}, {\em Regularizing the abstract convex
  program}, J. Math. Anal. Appl, 83 (1981), pp.~495--530.

\bibitem{CLO}
{\sc D.~Cox, J.~Little, and D.~O'Shea}, {\em Ideals, varieties, and algorithms:
  an introduction to computational algebraic geometry and commutative algebra},
  Springer, 2007.

\bibitem{DeKlerk2002}
{\sc E.~de~Klerk}, {\em Aspects of Semidefinite Programming: Interior Point
  Algorithms and Selected Applications}, Applied Optimization, Springer US,
  2006, \url{https://books.google.at/books?id=x5YMBwAAQBAJ}.

\bibitem{dostert2020exact}
{\sc M.~Dostert, D.~de~Laat, and P.~Moustrou}, {\em Exact semidefinite
  programming bounds for packing problems}, SIAM J. Optim., 31(2) (2021),
  pp.~1433--1458.

\bibitem{Drori:14}
{\sc Y.~Drori and M.~Teboulle}, {\em Performance of first-order methods for
  smooth convex minimization: A novel approach}, Mathematical Programming,
  145(1) (2014), pp.~451--482.

\bibitem{drusvyatskiy2017many}
{\sc D.~Drusvyatskiy and H.~Wolkowicz}, {\em The many faces of degeneracy in
  conic optimization}, Foundations and Trends in Optimization, 3 (2017),
  pp.~77--170.

\bibitem{eisenbud1988linear}
{\sc D.~Eisenbud}, {\em Linear sections of determinantal varieties}, American
  J. Math., 110 (1988), pp.~541--575.

\bibitem{goemans}
{\sc M.~Goemans and D.~Williamson}, {\em Improved approximation algorithms for
  maximum cuts and satisfiability problems using semidefinite programming},
  Journal of the ACM, 42 (1995), p.~1115–1145.

\bibitem{golub2013matrix}
{\sc G.~Golub and C.~Van~Loan}, {\em Matrix Computations}, Johns Hopkins
  Studies in the Mathematical Sciences, Johns Hopkins University Press, 2013,
  \url{https://books.google.at/books?id=X5YfsuCWpxMC}.

\bibitem{M2}
{\sc D.~Grayson and M.~Stillman}, {\em Macaulay2, a software system for
  research in algebraic geometry}.
\newblock Available at \url{http://www.math.uiuc.edu/Macaulay2/}.

\bibitem{harris1984symmetric}
{\sc J.~Harris and L.~W. Tu}, {\em On symmetric and skew-symmetric
  determinantal varieties}, Topology, 23 (1984), pp.~71--84.

\bibitem{harrison2007verifying}
{\sc J.~Harrison}, {\em Verifying nonlinear real formulas via sums of squares},
  in Theorem Proving in Higher Order Logics, Springer, 2007, pp.~102--118.

\bibitem{BertiniSDP}
{\sc J.~D. Hauenstein, A.~C.~L. Jr., S.~McPherson, and Y.~Zhang}, {\em
  Numerical algebraic geometry and semidefinite programming}, Results in
  Applied Mathematics, 11 (2021).

\bibitem{Helmberg2002}
{\sc C.~Helmberg}, {\em Semidefinite programming}, European Journal of
  Operational Research, 137 (2002), pp.~461--482,
  \url{https://doi.org/10.1016/S0377-2217(01)00143-6},
  \url{https://www.sciencedirect.com/science/article/pii/S0377221701001436}.

\bibitem{henrion2020moment}
{\sc D.~Henrion, M.~Korda, and J.~B. Lasserre}, {\em The Moment-SOS Hierarchy.
  Lectures In Probability, Statistics, Computational Geometry, Control And
  Nonlinear Pdes}, vol.~4, World Scientific, 2020.

\bibitem{HNS2014}
{\sc D.~Henrion, S.~Naldi, and M.~Safey El~Din}, {\em Real root finding for
  determinants of linear matrices}, J. of Symbolic Computation, 74 (2015),
  pp.~205--238, \url{https://doi.org/S0747717115000607},
  \url{http://www.sciencedirect.com/science/article/pii/S0747717115000607}.

\bibitem{HNS2015c}
{\sc D.~Henrion, S.~Naldi, and M.~Safey El~Din}, {\em Exact algorithms for
  linear matrix inequalities}, SIAM J. Optim., 26 (2016), pp.~2512--2539.

\bibitem{spectra}
{\sc D.~Henrion, S.~Naldi, and M.~Safey El~Din}, {\em {SPECTRA}: a {M}aple
  library for solving linear matrix inequalities in exact arithmetic},
  Optimization Methods and Software, 34 (2019), pp.~62--78.

\bibitem{HNS2015b}
{\sc D.~Henrion, S.~Naldi, and M.~Safey El~Din}, {\em Real root finding for low
  rank linear matrices}, Appl. Algebra Eng. Comm. Comput., 31 (2020),
  pp.~101--133.

\bibitem{henrion2021exact}
{\sc D.~Henrion, S.~Naldi, and M.~Safey El~Din}, {\em Exact algorithms for
  semidefinite programs with degenerate feasible set}, J. Symb. Comput., 104
  (2021), pp.~942--959.

\bibitem{KLYZ12}
{\sc E.~L. Kaltofen, B.~Li, Z.~Yang, and L.~Zhi}, {\em Exact certification in
  global polynomial optimization via sums-of-squares of rational functions with
  rational coefficients}, J. Symbolic Comput., 47 (2012), pp.~1--15,
  \url{https://doi.org/10.1016/j.jsc.2011.08.002},
  \url{http://dx.doi.org/10.1016/j.jsc.2011.08.002}.

\bibitem{klep2013exact}
{\sc I.~Klep and M.~Schweighofer}, {\em An exact duality theory for
  semidefinite programming based on sums of squares}, Mathematics of Operations
  Research, 38 (2013), pp.~569--590.

\bibitem{lasserre01globaloptimization}
{\sc J.-B. Lasserre}, {\em Global optimization with polynomials and the problem
  of moments}, SIAM J. Optim., 11 (2001), pp.~796--817,
  \url{https://doi.org/10.1137/S1052623400366802},
  \url{http://dx.doi.org/10.1137/S1052623400366802}.

\bibitem{Laurent2020}
{\sc M.~Laurent and F.~Vallentin}, {\em A course on semidefinite optimization:
  Draft lecture notes}, 2020.
\newblock Lecture notes, Spring 2020, Centrum Wiskunde \& Informatica and
  University of Cologne.

\bibitem{lourencco2021solving}
{\sc B.~Louren{\c{c}}o, M.~Muramatsu, and T.~Tsuchiya}, {\em Solving {SDP}
  completely with an interior point oracle}, Optimization Methods and Software,
   (2021), pp.~1--47.

\bibitem{monniaux2011generation}
{\sc D.~Monniaux and P.~Corbineau}, {\em On the generation of
  {P}ositivstellensatz witnesses in degenerate cases}, in International
  Conference on Interactive Theorem Proving, Springer, 2011, pp.~249--264.

\bibitem{naldi:tel-01212502}
{\sc S.~Naldi}, {\em {Exact algorithms for determinantal varieties and
  semidefinite programming}}, theses, {Université de Toulouse}, Sept. 2015,
  \url{https://tel.archives-ouvertes.fr/tel-01212502}.

\bibitem{naldiIssac2016integral}
{\sc S.~Naldi}, {\em Solving rank-constrained semidefinite programs in exact
  arithmetic}, J. Symb. Comput., 85 (2018), pp.~206--223.

\bibitem{naldi2020conic}
{\sc S.~Naldi and R.~Sinn}, {\em Conic programming: infeasibility certificates
  and projective geometry}, J. Pure Appl. Algebra,  (2020), p.~106605.

\bibitem{stu}
{\sc J.~Nie, K.~Ranestad, and B.~Sturmfels}, {\em The algebraic degree of
  semidefinite programming}, Mathematical Programming, 122 (2010),
  pp.~379--405, \url{https://doi.org/10.1007/s10107-008-0253-6},
  \url{http://dx.doi.org/10.1007/s10107-008-0253-6}.

\bibitem{Pan:00}
{\sc C.-T. Pan}, {\em On the existence and computation of rank-revealing {LU}
  factorizations}, Linear Algebra Appl., 316 (2000), pp.~199--222.

\bibitem{Papachristodoulou:05}
{\sc A.~Papachristodoulou and S.~Prajna}, {\em A tutorial on sum of squares
  techniques for systems analysis}, in American Control Conference, 2005,
  pp.~2686--2700.

\bibitem{Pataki:13}
{\sc G.~Pataki}, {\em Strong duality in conic linear programming: facial
  reduction and extended duals}, in Computational and Analytical Mathematics,
  vol.~50, Springer, New York, 2013, pp.~613--634.

\bibitem{Pataki2018}
{\sc G.~Pataki}, {\em Bad semidefinite programs: they all look the same}, SIAM
  Journal on Optimization, 27 (2017), pp.~146--172.

\bibitem{Pataki2020}
{\sc G.~Pataki}, {\em {On Positive Duality Gaps in Semidefinite Programming}},
  arXiv.org,  (2020), pp.~1--30, \url{http://arxiv.org/abs/1812.11796},
  \url{https://arxiv.org/abs/1812.11796v2}.

\bibitem{PermenterSelfDual:17}
{\sc F.~Permenter, H.~Friberg, and E.~Andersen}, {\em Solving conic
  optimization problems via self-dual embedding and facial reduction: a unified
  approach}, SIAM. J. Optim., 27 (2017), pp.~1257--1282.

\bibitem{ParriloPerm1}
{\sc F.~Permenter and P.~Parrilo}, {\em Partial facial reduction: simplified,
  equivalent {SDP}s via approximations of the {PSD} cone}, Mathematical
  Programming (Ser. A), 171 (2018), pp.~1--54.

\bibitem{PaPe}
{\sc H.~Peyrl and P.~Parrilo}, {\em Computing sum of squares decompositions
  with rational coefficients}, Theoretical Computer Science, 409 (2008),
  pp.~269 -- 281, \url{https://doi.org/10.1016/j.tcs.2008.09.025},
  \url{http://www.sciencedirect.com/science/article/pii/S0304397508006452}.
\newblock Symbolic-Numerical Computations.

\bibitem{Platzer:09}
{\sc A.~Platzer, J.-D. Quesel, and P.~R\"{u}mmer}, {\em Real world
  verification}, in CADE, 2009, pp.~485--501.

\bibitem{roux2018validating}
{\sc P.~Roux, Y.-L. Voronin, and S.~Sankaranarayanan}, {\em Validating
  numerical semidefinite programming solvers for polynomial invariants}, Formal
  Methods in System Design, 53 (2018), pp.~286--312.

\bibitem{SaSc03}
{\sc M.~{Safey El Din} and E.~Schost}, {\em Polar varieties and computation of
  one point in each connected component of a smooth real algebraic set}, in
  ISSAC'03, ACM, 2003, pp.~224--231.

\bibitem{SchorkGondzio:20}
{\sc L.~Schork and J.~Gondzio}, {\em Rank revealing gaussian elimination by the
  maximum volume concept}, Linear Algebra and its Applications, 592 (2020),
  pp.~1--19.

\bibitem{WakiMuramatsu:13}
{\sc H.~Waki and M.~Muramatsu}, {\em Facial reduction algorithms for conic
  optimization problems}, J. Optim. Theory. Appl., 158 (2013), pp.~188--215.

\bibitem{zhu2019sieve}
{\sc Y.~Zhu, G.~Pataki, and Q.~Tran-Dinh}, {\em Sieve-sdp: a simple facial
  reduction algorithm to preprocess semidefinite programs}, Math. Prog.
  Comput., 11 (2019), pp.~503--586.

\end{thebibliography}

\end{document}